\documentclass[reqno]{amsart}   
\usepackage{amssymb}
\usepackage{enumitem}
\usepackage[normalem]{ulem}

\usepackage{tikz}
\usepackage{mathabx}

\usetikzlibrary{decorations.pathmorphing, arrows, calc, matrix, arrows.meta}

\usepackage{subcaption}

\newcommand{\script}{\mathcal}
\newcommand{\parentheses}[1]{{\left( {#1} \right)}}

\newcommand{\p}{\parentheses}

\newcommand{\cl}{\mathrm{cl}}
\newcommand{\closure}[1]{\overline{#1}}

\newcommand{\Set}[1]{{\left\lbrace {#1} \right\rbrace}}
\newcommand{\singleton}{\Set}

\newcommand{\cardinality}[1]{{\left\lvert {#1} \right\rvert}}

\def\set#1:#2{\Set{{#1} \colon {#2}}}

\newcommand{\rooot}[1]{\text{r}\p{#1}}

\renewcommand{\times}{\square}

\DeclareMathOperator{\Ball}{Ball}

\tikzset{
square/.pic={
\node[draw,circle,scale=.3,fill] (b1) at (0,-2){};
\node[draw,circle,scale=.3,fill] (b2) at (0,-3){};
\node[draw,circle,scale=.3,fill] (start) at (0,-5){};
\node[draw,green,circle,scale=.3,fill] (green) at (-1,-2){};
\node[draw,yellow,circle,scale=.3,fill] (yellow) at (-1,-3){};
\draw[dotted] (0,0)--(b1);
\draw[dotted] (b2)--(start);
\draw (green)--(b1)--(b2)--(yellow);
\draw (start)--(0,-6);
\draw (-1,-6)--(1,-6)--(1,-7.6)--(1.3,-7.4)--(1.3,-8)--(-0.7,-8)--(-0.7,-6.4)--(-1,-6.6)--(-1,-6)--(1,-6);
}
}
\tikzset{
triangle/.pic={
\node[draw,circle,scale=.3,fill] (b1) at (0,-2){};
\node[draw,circle,scale=.3,fill] (b2) at (0,-3){};
\node[draw,circle,scale=.3,fill] (start) at (0,-5){};
\node[draw,yellow,circle,scale=.3,fill] (green) at (-1,-2){};
\node[draw,green,circle,scale=.3,fill] (yellow) at (-1,-3){};
\draw[dotted] (0,0)--(b1);
\draw[dotted] (b2)--(start);
\draw (green)--(b1)--(b2)--(yellow);
\draw (start)--(0,-6);
\draw (0,-6)--(1,-8)--(-1,-8)--(0,-6);
}
}
\tikzset{
decsquare/.pic={
\pic at (0,0) {square};
\pic[scale = 0.3,every node/.style={transform shape}] at (-0.6,-8) {square};
\pic[scale = 0.3, every node/.style={transform shape}] at (0.3,-8) {triangle};
\pic[scale = 0.3, every node/.style={transform shape}] at (1.2,-8) {square};
}
}
\tikzset{
dectriangle/.pic={
\pic at (0,0) {triangle};
\pic[scale = 0.3, every node/.style={transform shape}] at (-0.9,-8) {triangle};
\pic[scale = 0.3, every node/.style={transform shape}] at (0.0,-8) {triangle};
\pic[scale = 0.3, every node/.style={transform shape}] at (0.9,-8) {square};
}
}
\tikzset{
decdecsquare/.pic={
\pic at (0,0) {square};
\pic[scale = 0.3, every node/.style={transform shape}] at (-0.6,-8) {decsquare};
\pic[scale = 0.3, every node/.style={transform shape}] at (0.3,-8) {dectriangle};
\pic[scale = 0.3, every node/.style={transform shape}] at (1.2,-8) {decsquare};
}
}
\tikzset{
decdectriangle/.pic={
\pic at (0,0) {triangle};
\pic[scale = 0.3, every node/.style={transform shape}] at (-0.9,-8) {dectriangle};
\pic[scale = 0.3, every node/.style={transform shape}] at (0,-8) {dectriangle};
\pic[scale = 0.3, every node/.style={transform shape}] at (0.9,-8) {decsquare};
}
}

\newtheorem{theorem}{Theorem} \numberwithin{theorem}{section} 
\newtheorem{lemma}[theorem]{Lemma}
\newtheorem{cor}[theorem]{Corollary}
\newtheorem{prop}[theorem]{Proposition}

\newtheorem{defn}[theorem]{Definition}

\newtheorem{prob}{Problem}

\newtheorem*{claimm}{Claim}

\newtheorem*{mainresult}{Theorem \ref{t:one}}

\newcommand\numberthis{\addtocounter{equation}{1}\tag{\theequation}}

\newcommand{\N}{\mathbb{N}}

\begin{document}
\title[Non-reconstructible locally finite graphs]{Non-reconstructible locally finite graphs}

\author[N. Bowler, J. Erde, P. Heinig, F. Lehner, M. Pitz]{Nathan Bowler, Joshua Erde, Peter Heinig, Florian Lehner, Max Pitz}
\address{Department of Mathematics, Bundesstra{\ss}e 55, 20146 Hamburg, Germany}
\email{nathan.bowler@uni-hamburg.de, joshua.erde@uni-hamburg.de, \newline  heinig@ma.tum.de, mail@florian-lehner.net, max.pitz@uni-hamburg.de}



\begin{abstract}
Two graphs $G$ and $H$ are \emph{hypomorphic} if there exists a bijection $\varphi \colon V(G) \rightarrow V(H)$ such that $G - v \cong H - \varphi(v)$ for each $v \in V(G)$. A graph $G$ is \emph{reconstructible} if $H \cong G$ for all $H$ hypomorphic to $G$.

Nash-Williams proved that all locally finite connected graphs with a finite number $\geq 2$ of ends are reconstructible, and asked whether locally finite connected graphs with one end or countably many ends are also reconstructible.

In this paper we construct non-reconstructible connected graphs of bounded maximum degree with one and countably many ends respectively, answering the two questions of Nash-Williams about the reconstruction of locally finite graphs in the negative.
\end{abstract}

\maketitle
\section{Introduction}
Two graphs $G$ and $H$ are \emph{hypomorphic} if there exists a bijection $\varphi$ between their vertex sets such that the induced subgraphs $G - v$ and $H - \varphi(v)$ are isomorphic for each vertex $v$ of $G$. We say that a graph $G$ is \emph{reconstructible} if $H \cong G$ for every $H$ hypomorphic to $G$. The \emph{Reconstruction Conjecture}, a famous unsolved problem attributed to Kelly and Ulam, suggests that every finite graph with at least three vertices is reconstructible.

For an overview of results towards the Reconstruction Conjecture for finite graphs see the survey of Bondy and Hemminger \cite{BH77}. 
The corresponding reconstruction problem for infinite graphs is false: the countable regular tree $T_\infty$, and two disjoint copies of it (written as $T_\infty \cup T_\infty$) are easily seen to be a pair of hypomorphic graphs which are not isomorphic. This example, however, contains vertices of infinite degree. Regarding locally finite graphs, Harary, Schwenk and Scott \cite{HSS72} showed that there exists a non-reconstructible locally finite forest. However, they conjectured that the Reconstruction Conjecture should hold for locally finite trees. 
This conjecture has been verified for locally finite trees with at most countably many ends in a series of paper \cite{A81, BH74, T78}. However, very recently, the present authors have constructed a counterexample to the conjecture of Harary, Schwenk and Scott.

\begin{theorem}[Bowler, Erde, Heinig, Lehner, Pitz \cite{BEHLP17}]
\label{t:old}
There exists a non-recon-structible tree of maximum degree three.
\end{theorem}


The Reconstruction Conjecture has also been considered for general locally finite graphs. Nash-Williams \cite{NW87} showed that if $p \geq 3$ is an integer, then any locally finite connected graph with exactly $p$ ends is reconstructible; and in \cite{NW912} he showed the same is true for $p=2$. The case $p=2$ is significantly more difficult. Broadly speaking this is because every graph with $p \geq 3$ ends has some identifiable finite `centre', from which the ends can be thought of as branching out. A two-ended graph however can be structured like a double ray, without an identifiable `centre'. 

The case of $1$-ended graphs is even harder, and the following problems from a survey of Nash-Williams \cite{NW91}, which would generalise the corresponding results established for trees, have remained open.

\begin{prob}[Nash-Williams]\label{p:one}
Is every locally finite connected graph with exactly one end reconstructible?
\end{prob}

\begin{prob}[Nash-Williams]\label{p:count}
Is every locally finite connected graph with countably many ends reconstructible?
\end{prob}

In this paper, we extend our methods from \cite{BEHLP17} to construct examples showing that both of Nash-Williams' questions have negative answers. Our examples will not only be locally finite, but in fact have bounded degree.

\begin{theorem}\label{t:one}
There is a connected one-ended non-reconstructible graph with bounded maximum degree.
\end{theorem}

\begin{theorem}\label{t:count}
There is a connected countably-ended non-reconstructible graph with bounded maximum degree.
\end{theorem}

Since every locally finite connected graph has either finitely many, countably many or continuum many ends, Theorems \ref{t:old}, \ref{t:one} and \ref{t:count} together with the results of \cite{BH74,A81,T78} and the results of Nash-Williams \cite{NW87,NW912} provide a complete picture about what can be said about number of ends versus reconstruction:
\begin{itemize}
	\item A locally finite tree with at most countably many ends is reconstructible; but there are non-reconstructible locally finite trees with continuum many ends.
    \item A locally finite connected graph with at least two, but a finite number of ends is reconstructible; but there are non-reconstructible locally finite connected graphs with one, countably many, and continuum many ends respectively.
\end{itemize}

This paper is organised as follows: In the next section we give a short, high-level overview of our constructions which answer Nash-Williams' problems. In Sections~\ref{s:closure} and \ref{s:thickening}, we develop the technical tools necessary for our construction, and in Sections~\ref{s:proofone} and \ref{s:prooftwo}, we prove Theorems~\ref{t:one} and \ref{t:count}.

For standard graph theoretical concepts we follow the notation in \cite{D10}.

\section{Sketch of the construction}

In this section we sketch the main ideas of the construction in three steps. First, we quickly recall our construction of two hypomorphic, non-isomorphic locally finite trees from \cite{BEHLP17}. We will then outline how to adapt the construction to obtain a one-ended-, and a countably-ended counterexample respectively.

\subsection{The tree case}
This section contains a very brief summary of the much more detailed sketch from \cite{BEHLP17}. The strategy is to build trees $T$ and $S$ recursively, where at each step of the construction we ensure for some new vertex $v$ already chosen for $T$ that there is a corresponding vertex $w$ of $S$ with $T - v \cong S - w$, or vice versa. This will ensure that by the end of the construction, the trees we have built are hypomorphic. 

More precisely, at step $n$ we will construct subtrees $T_n$ and $S_n$ of our eventual trees, where some of the leaves of these subtrees have been coloured in two colours, say red and blue. 
We will only further extend the trees from these coloured leaves, and we will extend from leaves of the same colour in the same way. We also make sure that earlier partial isomorphisms between $T_n - v_i \cong S_n - w_i$ preserve leaf colours. Together, these requirements guarantee that earlier partial isomorphisms always extend to the next step.

The $T_n$ will be nested, and we will take $T$ to be the union of all of them; similarly the $S_n$ will be nested and we take $S$ to be the union of all of them. To ensure that $T$ and $S$ do not end up being isomorphic, we first ensure, for each $n$, that there is no isomorphism from $T_n$ to $S_n$. Our second requirement is that $T$ or $S$ beyond any coloured leaf of $T_n$ or $S_n$ begins with a long non-branching path, longer than any such path appearing in $T_n$ or $S_n$.  Together, this implies that $T$ and $S$ are not isomorphic. 

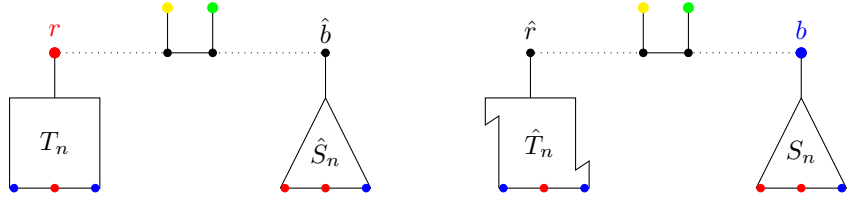
\begin{figure}[ht!]
\begin{subfigure}[t]{0.5\textwidth}
\centering
\begin{tikzpicture}[scale=0.6]

\node[text=red] at (1,3.5) {$r$};
\node at (7,3.5) {$\hat{b}$};

\node[draw, circle,scale=.3, fill] (Stop) at (7,3) {};
  
\draw (0,0) -- (2,0) -- (2,2) -- (0,2) -- (0,0);
\draw (1,2) -- (1,3);

\node at (1,1) {$T_n$};
  
\draw (6,0) -- (8,0) -- (7,2) -- (6,0);
\draw (7,2) -- (7,3);

\node at (7,0.8) {$\hat{S}_n$};

\node[draw, red, circle,scale=.3, fill] () at (1,.0) {};
\node[draw, blue, circle,scale=.3, fill] () at (0.1,.0) {};
\node[draw, blue, circle,scale=.3, fill] () at (1.9,.0) {};

\node[draw, red, circle,scale=.3, fill] () at (6.1,.0) {};
\node[draw, red, circle,scale=.3, fill] () at  (7,.0) {};
\node[draw, blue, circle,scale=.3, fill] () at (7.9,.0) {};
  
\node[draw, circle,scale=.3, fill] (gadget1) at (3.5,3) {};
\node[draw, circle,scale=.3, fill] (gadget2) at (4.5,3) {};
\draw (gadget1) -- (gadget2);

\node[draw, red, circle,scale=.4, fill] (Ttop) at (1,3) {};
    
\draw[dotted] (Ttop) -- (gadget1);
\draw[dotted] (Stop) -- (gadget2);
    
\draw(gadget1) -- (3.5,4);
\node[draw, yellow, circle,scale=.4, fill] (leaf1) at (3.5,4) {};
\node[draw, green, circle,scale=.4, fill] (leaf1) at (4.5,4) {};
\draw (gadget2) -- (leaf1);

\node[draw, red, circle,scale=.4, fill] (Ttop) at (1,3) {};
     
\end{tikzpicture}
\end{subfigure}%
\begin{subfigure}[t]{0.5\textwidth}
\centering
\begin{tikzpicture}[scale=0.6]

\node[draw, circle,scale=.3, fill] (Ttop) at (1,3) {};
\node[text=blue] at (7,3.5) {$b$};
\node at (1,3.5) {$\hat{r}$};

\node[draw, blue, circle,scale=.4, fill] (Stop) at (7,3) {};
  
\draw (0.3,0) -- (2.3,0) -- (2.3,0.6) -- (2,0.4)-- (2,2) -- (0,2) -- (0,1.4)--(0.3,1.6)--(0.3,0);
\draw (1,2) -- (1,3);

\node at (1.2,1) {$\hat{T}_n$};
  
\draw (6,0) -- (8,0) -- (7,2) -- (6,0);
\draw (7,2) -- (7,3);

\node at (7,0.8) {$S_n$};

\node[draw, red, circle,scale=.3, fill] () at (1.3,.0) {};
\node[draw, blue, circle,scale=.3, fill] () at (0.4,.0) {};
\node[draw, blue, circle,scale=.3, fill] () at (2.2,.0) {};

\node[draw, red, circle,scale=.3, fill] () at (6.1,.0) {};
\node[draw, red, circle,scale=.3, fill] () at  (7,.0) {};
\node[draw, blue, circle,scale=.3, fill] () at (7.9,.0) {};
  
\node[draw, circle,scale=.3, fill] (gadget1) at (3.5,3) {};
\node[draw, circle,scale=.3, fill] (gadget2) at (4.5,3) {};
\draw (gadget1) -- (gadget2);
    
\draw[dotted] (Ttop) -- (gadget1);
\draw[dotted] (Stop) -- (gadget2);
    
\draw(gadget1) -- (3.5,4);
\node[draw, yellow, circle,scale=.4, fill] (leaf1) at (3.5,4) {};
\node[draw, green, circle,scale=.4, fill] (leaf1) at (4.5,4) {};
\draw (gadget2) -- (leaf1);
\node[draw, blue, circle,scale=.4, fill] (Stop) at (7,3) {};
     
\end{tikzpicture}
\end{subfigure}
\caption{A first approximation of $T_{n+1}$ on the left, and $S_{n+1}$ on the right. All dotted lines are long non-branching paths.}
\label{sketchfig1}
\end{figure}

{\it Algorithm Stage One:} Suppose now that we have already constructed $T_n$ and $S_n$ and wish to construct $T_{n+1}$ and $S_{n+1}$. Suppose further that we are given a vertex $v$ of $T_n$ for which we wish to find a partner $w$ in $S_{n+1}$ so that $T - v$ and $S - w$ are isomorphic. We begin by building a tree $\hat{T}_n \not \cong T_n$ which has some vertex $w$ such that $T_n - v \cong \hat{T}_n - w$. This can be done by taking the components of $T_n - v$ and arranging them suitably around the new vertex $w$.

We will take $S_{n+1}$ to include $S_n$ and $\hat{T}_n$, with the copies of red and blue leaves in $\hat{T}_n$ also coloured red and blue respectively. As indicated on the right in Figure \ref{sketchfig1}, we add long non-branching paths to some blue leaf $b$ of $S_n$ and to some red leaf $\hat{r}$ of $\hat{T}_n$ and join these paths at their other endpoints by some edge $e_n$. We also join two new leaves $y$ and $g$ to the endvertices of $e_n$. We colour the leaf $y$ yellow and the leaf $g$ green. To ensure that $T_{n+1} - v \cong S_{n+1} - w$, we take $T_{n+1}$ to include $T_n$ together with a copy $\hat{S}_n$ of $S_n$, with its leaves coloured appropriately, and joined up in the same way, as indicated on the left in Figure \ref{sketchfig1}. Note that, whilst $\hat{S}_n$ and $S_n$ are isomorphic as graphs, we make a distinction as we want to lift the partial isomorphisms between $T_n - v_i \cong S_n - w_i$ to these new graphs, and our notation aims to emphasize the natural inclusions $T_n \subset T_{n+1}$ and $S_n \subset S_{n+1}$.

{\it Algorithm Stage Two:} We now have committed ourselves to two targets which are seemingly irreconcilable: first, we promised to extend in the same way at each red or blue leaf of $T_n$ and $S_n$, but we also need that $T_{n+1} - v \cong S_{n+1} - w$. The solution is to copy the same subgraph appearing beyond $r$ in Fig.~\ref{sketchfig1}, including its coloured leaves,  onto all the other red leaves of $S_n$ and $T_n$. Similarly we copy the subgraph appearing beyond the blue leaf $b$ of $S_n$ onto all other blue leaves of $S_n$ and $T_n$. In doing so, we create new red and blue leaves, and we will keep adding, step by step, further copies of the graphs appearing beyond $r$ and $b$ in Fig.~\ref{sketchfig1} respectively onto all red and blue leaves of everything we have constructed so far.

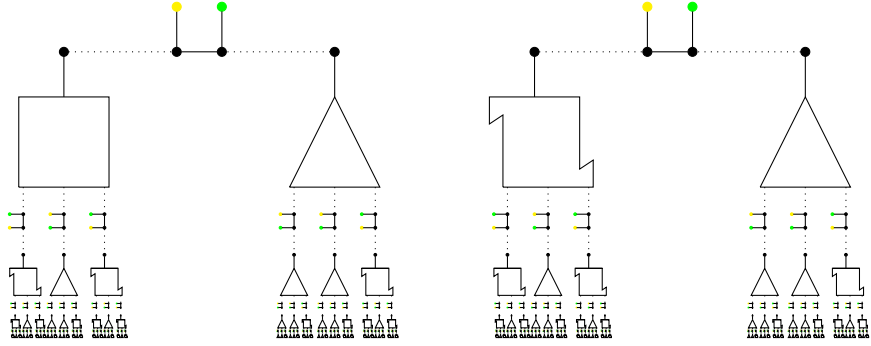
\begin{figure}[ht!]
\begin{subfigure}[t]{0.5\textwidth}
\centering
\begin{tikzpicture}[x=.5cm,y=.5cm, scale=1.2, every node/.style={transform shape}]
\node[draw, circle,scale=.3, fill] (Ttop) at (1,3) {};

\node[draw, circle,scale=.3, fill] (Stop) at (7,3) {};
  
\draw (0,0) -- (2,0) -- (2,2) -- (0,2) -- (0,0);
\draw (1,2) -- (1,3);

\draw (6,0) -- (8,0) -- (7,2) -- (6,0);
\draw (7,2) -- (7,3);

\pic[scale=0.3, every node/.style={transform shape}] at (0.1,0) {decdecsquare};
\pic[scale=0.3, every node/.style={transform shape}] at (1,0) {decdectriangle};
\pic[scale=0.3, every node/.style={transform shape}] at (1.9,0) {decdecsquare};

\pic[scale=0.3, every node/.style={transform shape}] at (6.1,0) {decdectriangle};
\pic[scale=0.3, every node/.style={transform shape}] at (7,0) {decdectriangle};
\pic[scale=0.3, every node/.style={transform shape}] at (7.9,0) {decdecsquare};
  
\node[draw, circle,scale=.3, fill] (gadget1) at (3.5,3) {};
\node[draw, circle,scale=.3, fill] (gadget2) at (4.5,3) {};
\draw (gadget1) -- (gadget2);
    
\draw[dotted] (Ttop) -- (gadget1);
\draw[dotted] (Stop) -- (gadget2);
    
\draw(gadget1) -- (3.5,4);
\node[draw, yellow, circle,scale=.3, fill] (leaf1) at (3.5,4) {};
\node[draw, green, circle,scale=.3, fill] (leaf1) at (4.5,4) {};
\draw (gadget2) -- (leaf1);
\end{tikzpicture}

\end{subfigure}%
\begin{subfigure}[t]{0.5\textwidth}
\centering
\begin{tikzpicture}[x=.5cm,y=.5cm, scale=1.2, every node/.style={transform shape}]

\node[draw, circle,scale=.3, fill] (Ttop) at (1,3) {};
;

\node[draw, circle,scale=.3, fill] (Stop) at (7,3) {};
  
\draw (0.3,0) -- (2.3,0) -- (2.3,0.6) -- (2,0.4)-- (2,2) -- (0,2) -- (0,1.4)--(0.3,1.6)--(0.3,0);
\draw (1,2) -- (1,3);

\draw (6,0) -- (8,0) -- (7,2) -- (6,0);
\draw (7,2) -- (7,3);

\pic[scale=0.3, every node/.style={transform shape}] at (0.4,0) {decdecsquare};
\pic[scale=0.3, every node/.style={transform shape}] at (1.3,0) {decdectriangle};
\pic[scale=0.3, every node/.style={transform shape}] at (2.2,0) {decdecsquare};

\pic[scale=0.3, every node/.style={transform shape}] at (6.1,0) {decdectriangle};
\pic[scale=0.3, every node/.style={transform shape}] at (7,0) {decdectriangle};
\pic[scale=0.3, every node/.style={transform shape}] at (7.9,0) {decdecsquare};

\node[draw, circle,scale=.3, fill] (gadget1) at (3.5,3) {};
\node[draw, circle,scale=.3, fill] (gadget2) at (4.5,3) {};
\draw (gadget1) -- (gadget2);
    
\draw[dotted] (Ttop) -- (gadget1);
\draw[dotted] (Stop) -- (gadget2);
    
\draw(gadget1) -- (3.5,4);
\node[draw, yellow, circle,scale=.3, fill] (leaf1) at (3.5,4) {};
\node[draw, green, circle,scale=.3, fill] (leaf1) at (4.5,4) {};
\draw (gadget2) -- (leaf1);

\end{tikzpicture}

\end{subfigure}
\caption{A sketch of $T_{n+1}$ and $S_{n+1}$ after countably many steps.}
\label{sketchfig2}
\end{figure}

After countably many steps we have dealt with all red and blue leaves, and it can be checked that both our targets are achieved. We take these new trees to be $S_{n+1}$ and $T_{n+1}$. They are non-isomorphic, as after removing all long non-branching paths, $T_{n+1}$ contains $T_n$ as a component, whereas $S_{n+1}$ does not.

\subsection{The one-ended case}
\label{s:1end}
To construct a one-ended non-reconstructible graph, we initially follow the same strategy as in the tree case and build locally finite graphs $G_n$ and $H_n$ and some partial hypomorphisms between them. Simultaneously, however, we will also build one-ended locally finite graphs of a grid-like form $F_n \times \N$ (the Cartesian product of a locally finite tree $F_n$ with a ray) which share certain symmetries with $G_n$ and $H_n$. These will allow us to glue $F_n \times \N$ onto both $G_n$ and $H_n$, in order to make them one-ended, without spoiling the partial hypomorphisms. Let us illustrate this idea by explicitly describing the first few steps of the construction.

We start with two non-isomorphic finite graphs $G_0$ and $H_0$, such that $G_0$ and $H_0$ each have exactly one red and one blue leaf. After stage one of our algorithm, our approximations to $G_1$ and $H_1$ as in Figure~\ref{sketchfig1} contain, in each of $G_0$, $\hat{H}_0$, $\hat{G}_0$ and $H_0$, one coloured leaf. In stage two, we add copies of these graphs recursively. It follows that the resulting graphs $G'_1$ and $H'_1$ have the global structure of a double ray, along which parts corresponding to copies of $G_0$, $\hat{H}_0$, $\hat{G}_0$ and $H_0$ appear in a repeating pattern. Crucially, however, each graph $G'_1$ and $H'_1$ has infinitely many yellow and green leaves, which appear in an alternating pattern extending to infinity in both directions along the double ray.

Consider the minor $F_1$ of $G'_1$ obtained by collapsing every subgraph corresponding to $G_0$, $\hat{H}_0$, $\hat{G}_0$ and $H_0$ to a single point. Write $\psi_G \colon G'_1 \to F_1$ for the quotient map. Then $F_1$ is a double ray with alternating coloured leaves hanging off it. Note that we could have started with $H'_1$ and obtained the same $F_1$. In other words, $F_1$ approximates the global structures of both $G'_1$ and $H'_1$. Consider the one-ended grid-like graph $F_1 \times \N$, where we let $F_1 \times \singleton{0}$ inherit the colours from $F_1$. We now form $G_1$ and $H_1$ by gluing $F_1 \times \N$ onto $G'_1$, by identifying corresponding coloured vertices $y$ and $\psi_G(y)$, and similarly for $H'_1$.\footnote{For technical reasons, in the actual construction we identify $\psi_G(y)$ with the corresponding base vertex of the leaf $y$ in $G'_1$. In this way the coloured leaves of $G'_1$ remain leaves, and we can continue our recursive construction.} Since the coloured leaves contained both ends of our graphs in their closure, the graphs $G_1$ and $H_1$ are now one-ended.

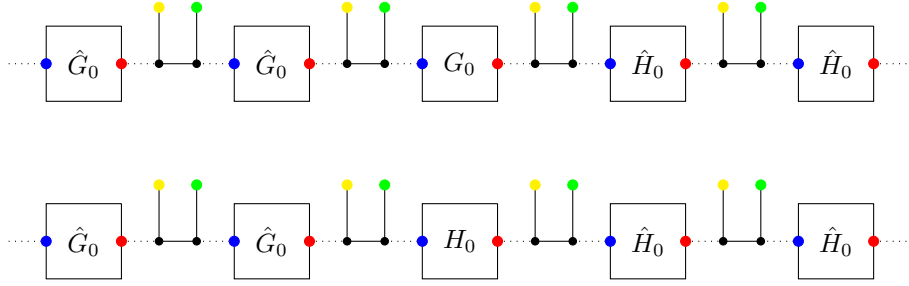
\begin{figure}
\begin{tikzpicture}[scale=0.5]

\draw (-5,0) -- (-4,0) (0,0)--(1,0)  (5,0)--(6,0) (-9,0)--(-10,0);

\draw[dotted] (-14,0)--(-13,0) (-11,0)--(-10,0) (-9,0) -- (-8,0) (-6,0) -- (-5,0)  (-4,0)--(-3,0)    (-1,0)--(0,0) (1,0)--(2,0)  (4,0)--(5,0) (6,0)--(7,0) (9,0)--(10,0);

\draw (-5,0) -- (-5,1.5)  (-4,0) -- (-4,1.5)  (0,0) -- (0,1.5)  (1,0) -- (1,1.5) (5,0) -- (5,1.5)  (6,0) -- (6,1.5) (-9,0)--(-9,1.5) (-10,0)--(-10,1.5);

\draw (-1,1) -- (-3,1) -- (-3,-1) -- (-1,-1) -- (-1,1);

\draw (-6,1) -- (-8,1) -- (-8,-1) -- (-6,-1) -- (-6,1);

\draw (4,1) -- (2,1) -- (2,-1) -- (4,-1) -- (4,1);

\draw (9,1) -- (7,1) -- (7,-1) -- (9,-1) -- (9,1);

\draw (-11,1) -- (-13,1) -- (-13,-1) -- (-11,-1) -- (-11,1);
\node[draw, circle,scale=.3, fill]  at (0,0) {};
\node[draw, circle,scale=.3, fill]  at (1,0) {};
\node[draw, blue, circle,scale=.4, fill]  at (2,0) {};
\node[draw,red, circle,scale=.4, fill]  at (4,0) {};
\node[draw, circle,scale=.3, fill]  at (5,0) {};
\node[draw, circle,scale=.3, fill]  at (6,0) {};
\node[draw,blue, circle,scale=.4, fill]  at (7,0) {};
\node[draw,red, circle,scale=.4, fill]  at (9,0) {};
\node[draw,red, circle,scale=.4, fill]  at (-1,0) {};
\node[draw, blue, circle,scale=.4, fill]  at (-3,0) {};
\node[draw, circle,scale=.3, fill]  at (-4,0) {};
\node[draw, circle,scale=.3, fill]  at (-5,0) {};
\node[draw, red, circle,scale=.4, fill]  at (-6,0) {};
\node[draw,blue, circle,scale=.4, fill]  at (-8,0) {};
\node[draw, circle,scale=.3, fill]  at (-9,0) {};
\node[draw, circle,scale=.3, fill]  at (-10,0) {};
\node[draw, red, circle,scale=.4, fill]  at (-11,0) {};
\node[draw,blue, circle,scale=.4, fill]  at (-13,0) {};

\node[draw, yellow, circle,scale=.4, fill]  at (0,1.5) {};
\node[draw, green, circle,scale=.4, fill]  at (1,1.5) {};

\node[draw, yellow, circle,scale=.4, fill] at (-5,1.5) {};
\node[draw, green, circle,scale=.4, fill]  at (-4,1.5) {};

\node[draw, yellow, circle,scale=.4, fill]  at (5,1.5) {};
\node[draw, green, circle,scale=.4, fill]  at (6,1.5) {};

\node[draw, yellow, circle,scale=.4, fill]  at (-10,1.5) {};
\node[draw, green, circle,scale=.4, fill]  at (-9,1.5) {};

\node at (-2,0) {$G_0$};

\node at (-7,0) {$\hat{G}_0$};

\node at (-12,0) {$\hat{G}_0$};

\node at (3,0) {$\hat{H}_0$};

\node at (8,0) {$\hat{H}_0$};

\end{tikzpicture}
\ \\
\vspace{1cm}
\begin{tikzpicture}[scale=0.5]

\draw (-5,0) -- (-4,0) (0,0)--(1,0)  (5,0)--(6,0) (-9,0)--(-10,0);

\draw[dotted] (-14,0)--(-13,0) (-11,0)--(-10,0) (-9,0) -- (-8,0) (-6,0) -- (-5,0)  (-4,0)--(-3,0)    (-1,0)--(0,0) (1,0)--(2,0)  (4,0)--(5,0) (6,0)--(7,0) (9,0)--(10,0);

\draw (-5,0) -- (-5,1.5)  (-4,0) -- (-4,1.5)  (0,0) -- (0,1.5)  (1,0) -- (1,1.5) (5,0) -- (5,1.5)  (6,0) -- (6,1.5) (-9,0)--(-9,1.5) (-10,0)--(-10,1.5);

\draw (-1,1) -- (-3,1) -- (-3,-1) -- (-1,-1) -- (-1,1);

\draw (-6,1) -- (-8,1) -- (-8,-1) -- (-6,-1) -- (-6,1);

\draw (4,1) -- (2,1) -- (2,-1) -- (4,-1) -- (4,1);

\draw (9,1) -- (7,1) -- (7,-1) -- (9,-1) -- (9,1);

\draw (-11,1) -- (-13,1) -- (-13,-1) -- (-11,-1) -- (-11,1);
\node[draw, circle,scale=.3, fill]  at (0,0) {};
\node[draw, circle,scale=.3, fill]  at (1,0) {};
\node[draw, blue, circle,scale=.4, fill]  at (2,0) {};
\node[draw,red, circle,scale=.4, fill]  at (4,0) {};
\node[draw, circle,scale=.3, fill]  at (5,0) {};
\node[draw, circle,scale=.3, fill]  at (6,0) {};
\node[draw,blue, circle,scale=.4, fill]  at (7,0) {};
\node[draw,red, circle,scale=.4, fill]  at (9,0) {};
\node[draw,red, circle,scale=.4, fill]  at (-1,0) {};
\node[draw, blue, circle,scale=.4, fill]  at (-3,0) {};
\node[draw, circle,scale=.3, fill]  at (-4,0) {};
\node[draw, circle,scale=.3, fill]  at (-5,0) {};
\node[draw, red, circle,scale=.4, fill]  at (-6,0) {};
\node[draw,blue, circle,scale=.4, fill]  at (-8,0) {};
\node[draw, circle,scale=.3, fill]  at (-9,0) {};
\node[draw, circle,scale=.3, fill]  at (-10,0) {};
\node[draw, red, circle,scale=.4, fill]  at (-11,0) {};
\node[draw,blue, circle,scale=.4, fill]  at (-13,0) {};

\node at (-2,0) {$H_0$};

\node at (-7,0) {$\hat{G}_0$};

\node at (-12,0) {$\hat{G}_0$};

\node at (3,0) {$\hat{H}_0$};

\node at (8,0) {$\hat{H}_0$};

\node[draw, yellow, circle,scale=.4, fill]  at (0,1.5) {};
\node[draw, green, circle,scale=.4, fill]  at (1,1.5) {};

\node[draw, yellow, circle,scale=.4, fill] at (-5,1.5) {};
\node[draw, green, circle,scale=.4, fill]  at (-4,1.5) {};

\node[draw, yellow, circle,scale=.4, fill]  at (5,1.5) {};
\node[draw, green, circle,scale=.4, fill]  at (6,1.5) {};

\node[draw, yellow, circle,scale=.4, fill]  at (-10,1.5) {};
\node[draw, green, circle,scale=.4, fill]  at (-9,1.5) {};

\end{tikzpicture}

\caption{A sketch of $G'_1$ (above) and $H'_1$ (below).}
\label{G2}
\end{figure}

It remains to check that our partial isomorphism $h_1 \colon G'_1-v_1 \to H'_1-w_1$ guaranteed by step two can be extended to $G_1-v_1 \to H_1-w_1$. This can be done essentially because of the following property: let us write $\script{L}(\cdot)$ for the set of coloured leaves. It can be checked that there is an automorphism $\pi_1 \colon F_1 \to F_1$ such that the diagram 
\begin{center}
\begin{tikzpicture}[scale=1.5]
  \matrix (m)
    [
      matrix of math nodes,
      row sep    = 3em,
      column sep = 7em
    ]
    {
      \script{L}(G'_1)              & \script{L}(H'_1) \\
      \script{L}(F_1) &  \script{L}(F_1)           \\
    };
  \path
    (m-1-1) edge [->] node [left] {\scriptsize{$\psi_G$}} (m-2-1)
    (m-1-1.east |- m-1-2)
      edge [->] node [above] {\scriptsize{$h_1 \restriction \script{L}(G'_1)$}} (m-1-2);
        \path
     (m-2-1) edge [->] node [above] {\scriptsize{$\pi_1  \restriction \script{L}(F_1)$}} (m-2-2);
        \path
     (m-1-2) edge [->] node [right] {\scriptsize{$\psi_H$}} (m-2-2);
\end{tikzpicture}
\end{center}
is colour-preserving and commutes. Hence, $\pi_1 \times \operatorname{id}$ is an automorphism of $F_1 \times \N$ which is compatible with our gluing procedure, so it can be combined with $h_1$ to give us the desired isomorphism.

We are now ready to describe the general step. Instead of describing $F_n$ as a minor of $G_n$, which no longer works na\"{\i}vely at later steps, we will directly build $F_n$ by recursion, so that it satisfies the properties of the above diagram.

Suppose at step $n$ we have constructed locally finite graphs $G_n$ and $H_n$, and also a locally finite tree $F_n$ where some leaves are coloured in one of two colours. Furthermore, suppose we have a family of isomorphisms 
\[
\script{H}_n = \set{h_x \colon G_n - x \to H_n - \varphi(x)}:{x \in X_n},
\]
for some subset $X_n \subset V(G_n)$, a family of isomorphisms $\Pi_n = \set{\pi_x \colon F_n \to F_n}:{x \in X_n}$, and colour-preserving bijections $\psi_{G_n} \colon \script{L}(G_n) \to \script{L}(F_n)$ and $\psi_{H_n} \colon \script{L}(H_n) \to \script{L}(F_n)$ such that the corresponding commutative diagram from above holds for each $x$. We construct $G'_{n+1}$ and $H'_{n+1}$ according to stages one and two of the previous algorithm. As before our isomorphisms $h_x$ will lift to isomorphisms between $G'_{n+1} - x$ and $H'_{n+1} -\varphi(x)$.

\begin{figure}[ht!]
\begin{tikzpicture}[scale=0.6]

\node[draw, red, circle,scale=.4, fill] (Ttop) at (1,3) {};
\node[text=red] at (1,3.5) {$\psi_{G_n}(r)$};
\node[text=blue] at (7,3.5) {$\psi_{H_n}(b)$};

\node[draw, blue, circle,scale=.3, fill] (Stop) at (7,3) {};
  
\draw (0,0) -- (2,0) -- (2,2) -- (0,2) -- (0,0);
\draw (1,2) -- (1,3);

\node at (1,1) {$F^G_n$};
  
\draw (6,0) -- (8,0) -- (8,2) -- (6,2) -- (6,0);
\draw (7,2) -- (7,3);

\node at (7,1) {$F^H_n$};

\node[draw, red, circle,scale=.3, fill] () at (1,.0) {};
\node[draw, blue, circle,scale=.3, fill] () at (0.1,.0) {};
\node[draw, blue, circle,scale=.3, fill] () at (1.9,.0) {};

\node[draw, red, circle,scale=.3, fill] () at (6.1,.0) {};
\node[draw, red, circle,scale=.3, fill] () at  (7,.0) {};
\node[draw, blue, circle,scale=.3, fill] () at (7.9,.0) {};
  
\node[draw, circle,scale=.3, fill] (gadget1) at (3.5,3) {};
\node[draw, circle,scale=.3, fill] (gadget2) at (4.5,3) {};
\draw (gadget1) -- (gadget2);
    
\draw (Ttop) -- (gadget1);
\draw (Stop) -- (gadget2);
\draw(gadget1) -- (3.5,4);
\node[draw, yellow, circle,scale=.4, fill] (leaf1) at (3.5,4) {};
\node[draw, green, circle,scale=.4, fill] (leaf1) at (4.5,4) {};
\draw (gadget2) -- (leaf1);
\node[draw, red, circle,scale=.4, fill] (Ttop) at (1,3) {};
\end{tikzpicture}

\caption{The auxiliary graph $\tilde{F}_n$.}
\label{sketch_aux}
\end{figure}
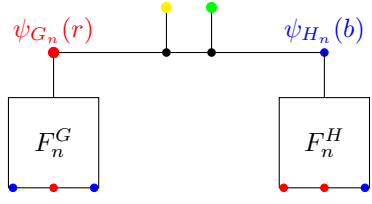

{\it Algorithm Stage Three.} As indicated in Figure~\ref{sketch_aux}, we take two copies $F_n^G$ and $F_n^H$ of $F_n$, and glue them together mimicking stage one of the algorithm, i.e.\ connect $\psi_{G_n}(r)$ in $F_n^G$ by a path of length three to $\psi_{H_n}(b)$ in $F_n^H$, and attach two new leaves coloured yellow and green in the middle of the path.  Call the resulting graph $\tilde{F}_n$. We then apply stage two of the algorithm to this graph, gluing again and again onto every blue vertex a copy of the graph of $\tilde{F}_n$ behind $\psi_{H_n}(b)$, and similarly for every red leaf, to obtain a tree $F_{n+1}$. Since this procedure is, in structural terms, so similar to the construction of $G'_{n+1}$ and $H'_{n+1}$, it can be shown that we do obtain a colour-preserving commuting diagram of the form

\begin{center}
\begin{tikzpicture}
  \matrix (m)
    [
      matrix of math nodes,
      row sep    = 3em,
      column sep = 7em
    ]
    {
      \script{L}(G'_{n+1})              & \script{L}(H'_{n+1}) \\
      \script{L}(F_{n+1}) &  \script{L}(F_{n+1})           \\
    };
  \path
    (m-1-1) edge [->] node [left] {\scriptsize{$\psi_{G_{n+1}}$}} (m-2-1)
    (m-1-1.east |- m-1-2)
      edge [->] node [above] {\scriptsize{$h_x \restriction \script{L}(G'_{n+1})$}} (m-1-2);
        \path
     (m-2-1) edge [->] node [above] {\scriptsize{$\pi_x  \restriction \script{L}(F_{n+1})$}} (m-2-2);
        \path
     (m-1-2) edge [->] node [right] {\scriptsize{$\psi_{H_{n+1}}$}} (m-2-2);
\end{tikzpicture}
\end{center}

As before, this means that we can indeed glue together $G'_{n+1}$ and $F_{n+1} \times \N$, and $H'_{n+1}$ and $F_{n+1} \times \N$ to obtain one-ended graphs $G_{n+1}$ and $H_{n+1}$ as desired.

At the end of our construction, after countably many steps, we have built two graphs $G$ and $H$ which are hypomorphic, and for the same reasons as in the tree case the two graphs will not be isomorphic. Further, since all $G_n$ and $H_n$ are one-ended, so will be $G$ and $H$.

\subsection{The countably-ended case} In order to produce hypomorphic graphs with countably many ends we follow the same procedure as for the one-ended case, except that we start with one-ended (non-isomorphic) graphs $G_0$ and $H_0$.

After the first and second stage of our algorithm, the resulting graphs $G'_1$ and $H'_1$ will again consist of infinitely many copies of $G_0$ and $H_0$ glued together along a double ray. After gluing $F_1 \times \mathbb{N}$ to these graphs as before, we obtain graphs with one thick end, with many coloured leaves tending to that end, as well as infinitely many thin ends, coming from the copies of $G_0$ and $H_0$, each of which contained a ray. These thin ends will eventually be rays, and so have no coloured leaves tending towards them. This guarantees that in the next step, when we glue $F_2 \times \mathbb{N}$ onto $G'_2$ and $H'_2$, the thin ends will not be affected, and that all the other ends in the graph will be amalgamated into one thick end.

Then, in each stage of the construction, the graphs $G_n$ and $H_n$ will have exactly one thick end, again with many coloured leaves tending towards it, and infinitely many thin ends each of which is eventually a ray. This property lifts to the graphs $G$ and $H$ constructed in the limit: they will have one thick end and infinitely many ends which are eventually rays. However, since $G$ and $H$ are countable, there can only be countably many of these rays. Hence the two graphs $G$ and $H$ have countably many ends in total, and as before they will be hypomorphic but not isomorphic.

\section{Closure with respect to promises}\label{s:closure}

\label{sec:defpromiseclosure}
A \emph{bridge} in a graph $G$ is an edge $e=\{x,y\}$ such that $x$ and $y$ lie in different components of $G-e$. Given a directed bridge $\vec{e}=\vec{xy}$ in some graph $G=(V,E)$, we denote by $G(\vec{e})$ the unique component of $G-e$ containing the vertex $y$. We think of $G(\vec{e})$ as a rooted graph with root $y$.

\begin{defn}[Promise structure]
A \emph{promise structure} $\script{P}=\p{G,\vec{P},\script{L}}$ is a triple consisting of: 
\begin{itemize}
\item a graph $G$,
\item $\vec{P}=\set{\vec{p}_i}:{i \in I}$ a set of directed bridges $\vec{P} \subset\vec{E}(G)$, and
\item $\script{L}=\set{L_i}:{i \in I}$ a set of pairwise disjoint sets of leaves of $G$.
\end{itemize}
We insist further that, if the component $G(\vec{p_i})$ consists of a single leaf $c \in L_j$, then $i=j$.
\end{defn}
Often, when the context is clear, we will not make a distinction between $\script{L}$ and the set $\bigcup_i L_i$, for notational convenience. 

We call an edge $\vec{p}_i \in \vec{P}$ a \emph{promise edge}, and leaves $\ell \in L_i$ \emph{promise leaves}. A promise edge $\vec{p_i} \in \vec{P}$ is called a \emph{placeholder-promise} if the component $G(\vec{p_i})$ consists of a single leaf $c \in L_i$, which we call a \emph{placeholder-leaf}. We write 
\[
\script{L}_p =\set{ L_i}:{\vec{p_i} \text{ a placeholder-promise}} \text{    and    } \script{L}_q = \script{L} \setminus \script{L}_p.
\]

Given a leaf $\ell$ in $G$, there is a unique edge $q_{\ell} \in E(G)$ incident with $\ell$, and this edge has a natural orientation $\vec{q_{\ell}}$ towards $\ell$. Informally, we think of $\ell \in L_i$ as the `promise' that if we extend $G$ to a graph $H \supset G$, we will do so in such a way that $H(\vec{q_{\ell}}) \cong H(\vec{p_i})$.

\begin{defn}[Leaf extension]
Given an inclusion $H\supseteq G$ of graphs and a set $L$ of leaves of $G$, $H$ is called a \emph{leaf extension}, or more specifically an \emph{$L$-extension, of $G$}, if:
\begin{itemize}
\item every component of $H$ contains precisely one component of $G$, and
\item every component of $H - G$ is adjacent to a unique vertex $\ell$ of $G$, and we have $\ell \in L$.
\end{itemize}
\end{defn}

In \cite{BEHLP17}, given a promise structure $\script{P}=\p{G,\vec{P},\script{L}}$, it is shown how to construct a graph $\cl(G) \supset G$ which has the following properties.

\begin{prop}[Closure w.r.t a promise structure, cf.~{\cite[Proposition 3.3]{BEHLP17}}]\label{p:closure}
Let $G$ be a graph and let $\script{P} = \p{G,\vec{P},\script{L}}$ be a promise structure. Then there is a graph $\cl(G)$, called the \emph{closure of $G$ with respect to $\script{P}$}, such that:
\begin{enumerate}[series=properties,label={\upshape(cl.\arabic{*})}]
\item\label{i:extension} $\cl(G)$ is an $\script{L}_q$-extension of $G$, 
\item\label{i:keepspromises}  for every $\vec{p}_i \in \vec{P}$ and all $\ell \in L_i$, $$\cl(G)(\vec{p}_i) \cong \cl(G)(\vec{q}_{\ell})$$ are isomorphic as rooted graphs.
\end{enumerate}
\end{prop}

Since the existence of $\cl(G)$ is crucial to our proof, we briefly remind the reader how to construct such a graph. As a first approximation, in order to try to achieve (\ref{i:keepspromises}), we glue a copy of the component $G(\vec{p}_i)$ onto each leaf $\ell \in L_i$, for each $i \in I$. We call this the \emph{1-step extension $G^{(1)}$} of $G$. If there were no promise leaves in the component $G(\vec{p}_i)$, then the promises in $L_i$ would be satisfied. However, if there are, then we have grown $G(\vec{p}_i)$ by adding copies of various $G(\vec{p}_j)$s behind promise leaves appearing in $G(\vec{p}_i)$. 

However, remembering all promise leaves inside the newly added copies of $G(\vec{p}_i)$ we glued behind each $\ell \in L_i$, we continue this process indefinitely, growing the graph one step at a time by gluing copies of (the original) $G(\vec{p}_i)$ to promise leaves $\ell'$ which have appeared most recently as copies of $\ell \in L_i$. After a countable number of steps the resulting graph $\cl(G)$ satisfies Proposition \ref{p:closure}. We note also that the maximum degree of $\cl(G)$ equals that of $G$.

\begin{defn}[Promise-respecting map]
\label{promiserespectingmap}
Let $G$ be a graph, $\script{P} = \p{G,\vec{P},\script{L}}$ be a promise structure on $G$, and let $T_1$ and $T_2$ be two components of $G$.

Given $x \in T_1$ and $y \in T_2$, a bijection $\varphi \colon T_1 - x \rightarrow T_2 - y$ is \emph{$\vec{P}$-respecting} (with respect to $\script{P}$) if the image of $L_i \cap T_1$ under $\varphi$ is $L_i \cap T_2$ for all $i$.
\end{defn}

We can think of $\script{P}$ as defining a $|\vec{P}|$-colouring on some sets of leaves. Then a mapping is $\vec{P}$-respecting if it preserves leaf colours.

Suppose that $\vec{p_i}$ is a placeholder promise, and $G = H^{(0)} \subseteq H^{(1)} \subseteq \cdots$ is the sequence of $1$-step extensions whose direct limit is $\cl(G)$. Then, if we denote by $L^{(n)}_i$ the set of promise leaves associated with $\vec{p_i}$ in $H^{(n)}$, it follows that $L_i^{(n)} \supseteq L_{i}^{(n-1)}$ since $G(\vec{p_i})$ is just a single vertex $c_i \in L_i$. For every placeholder promise $\vec{p}_i \in \vec{P}$, we define $\cl(L_i)= \bigcup_n L_i^{(n)}$.

\begin{defn}[Closure of a promise structure]
\label{def_closurestructure} The \emph{closure} of the promise structure $\p{G, \vec{P}, \script{L}}$ is the promise structure $\cl(\script{P}) = \p{\cl(G), \cl(\vec{P}),\cl(\script{L})}$, where:
\begin{itemize}
\item $\cl(\vec{P})=\set{\vec{p}_i}:{\text{$\vec{p_i} \in \vec{P}$ is a placeholder-promise}}$, 
\item $\cl(\script{L}) = \{\cl(L_i) \colon \text{$\vec{p_i} \in \vec{P}$ is a placeholder-promise}\}$.
\end{itemize}
\end{defn}

\begin{prop}[{\cite[Proposition 3.3]{BEHLP17}}]
\label{p:alfnbadjklfng}
Let $G$ be a graph and let $\p{G,\vec{P},\script{L}}$ be a promise structure. Then $\cl(G)$ satisfies:
\begin{enumerate}[resume=properties,label={\upshape(cl.\arabic{*})}]
\item\label{i:keepslabelledpromises}  for every $\vec{p_i} \in \vec{P}$ and every $\ell \in L_i$, $$\cl(G)(\vec{p_i}) \cong \cl(G)(\vec{q_{\ell}})$$ are isomorphic as rooted graphs, and this isomorphism is $\cl(\vec{P})$-respecting with respect to $\cl(\script{P})$.
\end{enumerate}
\end{prop}

It is precisely this property \ref{i:keepslabelledpromises} of the promise closure that will allow us to maintain partial hypomorphisms during our recursive construction.

The last two results of this section serve as preparation for growing $G_{n+1}$, $H_{n+1}$ and $F_{n+1}$ `in parallel', as outlined in the third stage of the algorithm in \S\ref{s:1end}. If $\script{L} =  \set{L_i}:{i \in I}$ and $\script{L'} =  \set{L'_i}:{i \in I}$, we say a map $\psi : \bigcup \script{L} \rightarrow \bigcup \script{L'}$ is \emph{colour-preserving} if  $\psi(L_i) \subseteq L'_i$ for every $i$.

\begin{lemma}\label{l:blabla}
Let $\p{G,\vec{P},\script{L}}$ and $\p{G',\vec{P'},\script{L'}}$ be promise structures, and let $G = H^{(0)} \subseteq H^{(1)} \subseteq \cdots$ and $G' = H'^{(0)} \subseteq H'^{(1)} \subseteq \cdots$ be 1-step extensions approximating their respective closures. 

Assume that $\vec{P}=\Set{\vec{p}_1,\ldots,\vec{p}_k}$ and $\vec{P'}=\Set{\vec{r}_1,\ldots,\vec{r}_k}$, and that there is a colour-preserving bijection
$$\psi \colon \bigcup\script{L} \to \bigcup\script{L}'$$
such that (recall that $\script{L}(\cdot)$ is the set of leaves of a graph that are in $\script{L}$)
$$\psi \restriction G(\vec{p}_i) \colon  \script{L}(G(\vec{p}_i)) \to  \script{L}'(G'(\vec{r}_i))$$
is still a colour-preserving bijection for all $\vec{p}_i \in \vec{P}$.

Then for each $i \leq k$ there is a sequence of colour-preserving bijections 
$$\alpha^i_n \colon \script{L}\p{H^{(n)}(\vec{p_i})} \to \script{L}'\p{H'^{(n)}(\vec{r_i})}$$
such that $\alpha^i_{n+1}$ extends $\alpha^i_n$.
\end{lemma}

\begin{proof}
Fix $i$. We proceed by induction on $n$. Put $\alpha^i_0 := \psi \restriction G(\vec{p}_i)$.

Now suppose that $\alpha^i_n$ exists. To form $H^{(n+1)}(\vec{p_i})$, we glued a copy of $G(\vec{p_j})$ to each $\ell \in L^{(n)}_j \cap H^{(n)}(\vec{p_i})$ for all $j \leq k$, and to construct $H'^{(n+1)}(\vec{r_i})$, we glued a copy of $G'(\vec{r_j})$ to each $\ell' \in L'^{(n)}_j \cap H'^{(n)}(\vec{r_i})$ for all $j \leq k$, in both cases keeping all copies of promise leaves.

By assumption, the second part can be phrased equivalently as: we glued on a copy of $G'(\vec{r_j})$ to each $\alpha^i_n(\ell)$ for $\ell\in L^{(n)}_j \cap H^{(n)}(\vec{r_i})$. Thus, we can now combine the bijections $\alpha^i_n(\ell)$ with all the individual bijections $\psi$ between all newly added $G(\vec{p_j})$ and $G'(\vec{r_j})$ to obtain a bijection $\alpha^i_{n+1}$ as desired.
\end{proof}

\begin{cor}
\label{cor_blablablabla}
In the above situation, for each $i$ there is a colour-preserving bijection $\alpha^i$ between $\script{L}\p{\cl(G)(\vec{p}_i)}$ and $\script{L}'\p{\cl(G')(\vec{r}_i)}$ with respect to the promise closures $\cl(\script{P})$ and $\cl(\script{P}')$.
\end{cor}
\begin{proof}
Put $\alpha^i = \bigcup_n \alpha^i_n$. Because all $\alpha^i_n$ respected all colours, they respect in particular the placeholder promises which make up $\cl(\script{P})$ and $\cl(\script{P}')$.
\end{proof}

\section{Thickening the graph}
\label{s:thickening}

In this section, we lay the groundwork for the third stage of our algorithm, as outlined in \S\ref{s:1end}. Our aim is to clarify how gluing a one-ended graph $F$ onto a graph $G$ affects automorphisms and the end-space of the resulting graph.

\begin{defn}[Gluing sum]
\label{gluingsum}
Given two graphs $G$ and $F$, and a bijection $\psi$ with $\operatorname{dom}(\psi) \subseteq V(G)$ and $\operatorname{ran}(\psi) \subseteq V(F)$, the \emph{gluing sum of $G$ and $F$ along $\psi$}, denoted by $G \oplus_{\psi} F$, is the quotient graph $\p{G \cup F}/\sim$ where $v \sim \psi(v)$ for all $v \in \operatorname{dom}(\psi)$.
\end{defn}

Our first lemma of this section explains how a partial isomorphism from $G_{n} - x$ to $H_{n} - \phi(x)$ in our construction can be lifted to the gluing sum of $G_{n}$ and $H_n$ with a graph $F$ respectively.

\begin{lemma}\label{l:extend}
Let $G$, $H$ and $F$ be graphs, and consider two gluing sums $G \oplus_{\psi_G} F$ and $H \oplus_{\psi_H} F$ along partial bijections $\psi_G$ and $\psi_H$. Suppose there exists an isomorphism $h\colon G - x \rightarrow H - y$ that restricts to a bijection between $\operatorname{dom}(\psi_G)$ and $\operatorname{dom}(\psi_H)$.

Then $h$ extends to an isomorphism $(G \oplus_{\psi_G} F) - x \to (H \oplus_{\psi_H} F) - y$ provided there is an automorphism $\pi$ of $F$ such that $\pi \circ \psi_G (v) = \psi_H \circ h(v)$ for all $v \in \operatorname{dom}(\psi_G)$.

\end{lemma}
\begin{proof}
We verify that the map
$$ \hat{h} \colon \p{G \oplus_{\psi_G} F} - x \rightarrow \p{H \oplus_{\psi_H}F} - y, \quad v \mapsto \begin{cases} h(v) & \text{ if } v \in G-x, \text{ and} \\ \pi(v) & \text{ if } v \in F \end{cases} $$
is a well-defined isomorphism. It is well-defined, since if $v \sim \psi_G(v)$ in $G \oplus_{\psi_G} F$, then $\hat{h}(v) \sim \hat{h}(\psi_G(v))$ in $H \oplus_{\psi_H} F$ by assumption on $\pi$. Moreover, since $h$ and $\pi$ are isomorphisms, it follows that $\hat{h}$ is an isomorphism, too.
\end{proof}

For the remainder of this section, all graphs are assumed to be locally finite. A \emph{ray} in a graph $G$ is a one-way infinite path. Given a ray $R$, then for any finite vertex set $S \subset V(G)$ there is a unique component $C(R,S)$ of $G - S$ containing a tail of $R$. An \emph{end} in a graph is an equivalence class of rays under the relation
\[
R \sim R' \Leftrightarrow \text{for every finite vertex set }S \subset V(G) \text{ we have } C(R,S) = C(R',S).
\]
We denote by $\Omega(G)$ the set of ends in the graph $G$, and write $C(\omega,S):= C(R,S)$ with $R \in \omega$. Let $\Omega(\omega,S) = \set{\omega'}:{C(\omega',S) = C(\omega,S)}$. The singletons $\singleton{v}$ for $v \in V(G)$ and sets of the form $C(\omega,S) \cup \Omega(\omega,S)$ generate a compact metrizable topology on the set $V(G) \cup \Omega(G)$, which is known in the literature as $|G|$.\footnote{Normally $|G|$ is defined on the $1$-complex of $G$ together with its ends, but for our purposes it will be enough to just consider the subspace $V(G) \cup \Omega(G)$. See the survey paper of Diestel \cite{D11} for further details.} This topology allows us to talk about the closure of a set of vertices $X \subset V(G)$, denoted by $\closure{X}$. Write $\partial (X) = \closure{X} \setminus X = \closure{X} \cap \Omega(X)$ for the boundary of $X$: the collection of all ends in the closure of $X$. Then an end $\omega \in \Omega(G)$ lies in $\partial (X)$ if and only if for every finite vertex set $S \subset V(G)$, we have $\cardinality{X  \cap C(\omega,S)} = \infty$. Therefore $\Omega(G) = \partial (X)$ if and only if for every finite vertex set $S \subset V(G)$, every infinite component of $G - S$ meets $X$ infinitely often. In this case we say that $X$ is \emph{dense} for $\Omega(G)$.

Finally, an end $\omega \in \Omega(G)$ is \emph{free} if for some $S$, the set $\Omega(\omega,S) = \singleton{\omega}$. Then $\Omega'(G)$ denotes the \emph{non-free} (or {limit}-)ends. Note that $\Omega'(G)$ is a closed subset of $\Omega(G)$.

\begin{lemma}\label{l:oneend}
For locally finite connected graphs $G$ and $F$, consider the gluing sum $G \oplus_{\psi} F$ for a partial bijection $\psi$. If $F$ is one-ended and $\operatorname{dom}(\psi)$ is infinite, then $\Omega(G \oplus_\psi F) \cong \Omega (G) / \partial (\operatorname{dom}(\psi))$.
\end{lemma}
\begin{proof}
Note first that for locally finite graphs $G$ and $F$, also $G \oplus_{\psi} F$ is locally finite. Observe further that all rays of the unique end of $F$ are still equivalent in $G \oplus_\psi F$, and so $G \oplus_\psi F$ has an end $\hat{\omega}$ containing the single end of $F$.

We are going to define a continuous surjection $f \colon \Omega(G) \to \Omega(G \oplus_\psi F)$ with the property that $f$ has precisely one non-trivial fibre, namely $f^{-1}(\hat{\omega}) = \partial \p{\operatorname{dom}(\psi)}$. It then follows from definition of the quotient topology that $f$ induces a continuous bijection from the compact space $\Omega (G) / \partial \p{\operatorname{dom}(\psi)}$ to the Hausdorff space $\Omega(G \oplus_\psi F)$, which, as such, is necessarily a homeomorphism.

The mapping $f$ is defined as follows. Given an end $\omega \in \Omega(G) \setminus \partial \p{\operatorname{dom}(\psi)}$, there is a finite $S \subset V(G)$ such that $C(\omega,S) \cap \operatorname{dom}(\psi) = \emptyset$, and so $C=C(\omega,S)$ is also a component of $\p{G \oplus_{\psi} F} - S$, which is disjoint from $F$. Define $f$ to be the identity between $\Omega(G) \cap \closure{C}$ and $\Omega(G \oplus_\psi F) \cap \closure{C}$, while for all remaining ends $\omega \in \Omega(G) \cap \closure{\operatorname{dom}(\psi)}$, we put $f(\omega) = \hat{\omega}$. 

To see that this assignment is continuous at $\omega \in \Omega(G) \cap \closure{\operatorname{dom}(\psi)}$, it suffices to show that $C:=C(\omega, S) \subset G - S$ is a subset of $C':=C(\hat{\omega},S) \subset \p{G \oplus_\psi F} -S$ for any finite set $S \subset G \oplus_{\psi} F$. To see this inclusion, note that by choice of $\omega$, we have $\cardinality{ \operatorname{dom}(\psi) \cap C} = \infty$. 
At the same time, since $F$ is both one-ended and locally finite, $F - S$ has precisely one infinite component $D$ and $F - D$ is finite, so as $\psi$ is a bijection, there is $v \in \operatorname{dom}(\psi) \cap C$ with $\psi(v) \in D$ (in fact, there are infinitely many such $v$).
Since $v$ and $\psi(v)$ get identified in $G \oplus_{\psi} F$, we conclude that $C \cup D$ is connected in $\p{G \oplus_{\psi} F} - S$, and hence that $C \cup D \subset C'$ as desired.

Finally, to see that $f$ is indeed surjective, note first that the fact that $\operatorname{dom}(\psi)$ is infinite implies that $\closure{\operatorname{dom}(\psi)} \cap \Omega(G) \neq \emptyset$, and so $\hat{\omega} \in \operatorname{ran}(f)$. Next, consider an end $\omega \in \Omega(G \oplus_\psi F)$ different from $\hat{\omega}$. Find a finite separator $S \subset V(G \oplus_\psi F)$ such that $C(\omega,S) \neq C(\hat{\omega},S)$. It follows that $\operatorname{dom}(\psi) \cap C(\omega,S) $ is finite. So there is a finite $S' \supseteq S$ such that $C:=C(\omega,S') \neq C(\hat{\omega},S')$ and $\operatorname{dom}(\psi) \cap C= \emptyset $. So by definition,  $f$ is a bijection between $\Omega(G) \cap \closure{C}$ and $\Omega(G \oplus_\psi F) \cap \closure{C}$, so $\omega \in \operatorname{ran}(f)$.
\end{proof}

\begin{cor}\label{c:oneend}
Under the above assumptions, if $\operatorname{dom}(\psi)$ is dense for $\Omega(G)$, then $G \oplus_{\psi} F$ is one-ended.
\end{cor}
\begin{cor}\label{c:countend}
Under the above assumptions, if $\closure{\operatorname{dom}(\psi)} \cap \Omega(G) = \Omega'(G)$, then $G \oplus_{\psi} F$ has at most one non-free end.
\end{cor}

We remark that more direct proofs for Corollaries~\ref{c:oneend} and \ref{c:countend} can be given that do not need the full power of Lemma~\ref{l:oneend}.

\section{The construction}
\label{s:proofone}

\subsection{Preliminary definitions} In the precise statement of our construction in \S\ref{s:bandf}, we are going to employ the following notation.

\begin{defn}[Bare path, maximally bare path]
A path $P=v_0,v_1,\ldots,v_n$ in a graph $G$ is called a \emph{bare path} if $\deg_G(v_i)=2$ for all internal vertices $v_i$ for $0< i < n$. The path $P$ is a \emph{maximal bare path} (or \emph{maximally bare}) if in addition $\deg_G(v_0) \neq 2 \neq\deg_G(v_n)$. An infinite path $P=v_0,v_1,v_2, \ldots$ is \emph{maximally bare} if $\deg_G(v_0)\neq 2$ and $\deg_G(v_i)=2$ for all $i \geq 1$. 
\end{defn}

\begin{defn}[Bare-spectrum]
The \emph{bare-spectrum} of $G$ is
$$ \Sigma(G):= \set{k \in \N}:{G \text{ contains an maximally bare path of length } k }. $$
If $\Sigma(G)$ is finite, we let $\sigma_0(G) = \max \Sigma(G)$ and  $\sigma_1(G) =  \max \p{ \Sigma(G) \setminus \singleton{\sigma_0(G)}}$.
\end{defn}

\begin{lemma}\label{l:miibound}
Let $e$ be an edge of a locally finite graph $G$. If $\Sigma(G)$ is finite, then $\Sigma(G - e)$ is finite.
\end{lemma}
\begin{proof}
Observe first that every vertex of degree $\leq 2$ in any graph can lie on at most one maximally bare path.

We now claim that for an edge $e = xy$, there are at most two finite maximally bare paths in $G-e$ which are not subpaths of finite maximally bare paths of $G$.

Indeed, if $\deg{x}=3$ in $G$, then $x$ can now be the interior vertex of one new finite maximally bare path in $G-e$. And if $\deg{x}=2$ in $G$, then $x$ can now be end-vertex of one new finite maximally bare path in $G-e$ (this is relevant if $x$ lies on an infinite maximally bare path of $G$). The argument for $y$ is the same, so the claim follows.
\end{proof}

\begin{defn}[Spectrally distinguishable]
Given two graphs $G$ and $H$, we say that $G$ and $H$ are \emph{spectrally distinguishable} if there is some $k \geq 3$ such that $k \in \Sigma(G) \triangle \Sigma(H) = \Sigma(G) \setminus \Sigma(H) \cup  \Sigma(H) \setminus \Sigma(G)$.
\end{defn}

Note that being spectrally distinguishable is a strong certificate for being non-isomorphic.

\begin{defn}[$k$-ball]
For $G$ a subgraph of $H$, and $k > 0$, the $k$-ball $\Ball_H(G, k)$ is the induced subgraph of $H$ on the set of vertices at distance at most $k$ of some vertex of $G$.
\end{defn}

\begin{defn}[Proper bare extension; infinite growth]
Let $G$ be a graph, $B$ a subset of leaves of $G$, and $H$ a component of $G$.
\begin{itemize}
\item  A graph $\hat{G} \supset H$ is an {\em bare extension} of $H$ at $B$ to length $k$ if $\Ball_{\hat{G}}(H,k)$ can be obtained from $H$ by adjoining, at each vertex $\ell \in B \cap V(H)$, a new path of length $k$ starting at $\ell$, and a new leaf whose only neighbour is $\ell$.\footnote{We note that this is a slightly different definition of an bare extension to that in \cite{BEHLP17}.}
\item A leaf $\ell$ in a graph $G$ is \emph{proper} if the unique neighbour of $\ell$ in $G$ has degree $\geq 3$. An bare extension is called \emph{proper} if every leaf in $B$ is proper.
\item An bare extension  $\hat{G}$ of $G$ is \emph{of infinite growth} if every component of $\hat{G} - G$ is infinite.
\end{itemize}
\end{defn}

\subsection{The back-and-forth construction}\label{s:bandf}
Our aim in this section is to prove our main theorem announced in the introduction.

\begin{mainresult}
There are two hypomorphic connected one-ended infinite graphs $G$ and $H$ with maximum degree five such that $G$ is not isomorphic to $H$.
\end{mainresult}

To do this we shall recursively construct, for each $n \in \N$, 
\begin{itemize}
\item disjoint rooted connected graphs $G_n$ and $H_n$, 
\item disjoint sets $R_n$ and $B_n$ of proper leaves of the graph $G_n \cup H_n$,
\item trees $F_n$,
\item disjoint sets $R'_n$ and $B'_n$ of leaves of $F_n$,
\item bijections $\psi_{G_n} \colon V(G_n) \cap \p{R_n \cup B_n} \to R'_n \cup B'_n$ and \newline $\psi_{H_n} \colon V(H_n) \cap \p{R_n \cup B_n} \to R'_n \cup B'_n$,
\item finite sets $X_n \subset V(G_n)$ and $Y_n \subset V(H_n)$, and bijections $\varphi_n \colon X_n \to Y_n$,
\item a family of isomorphisms $\script{H}_n = \set{h_{n,x} \colon G_n - x \to H_n - \varphi_n(x)}:{x \in X_n}$,
\item a family of automorphisms $\Pi_n = \set{\pi_{n,x} \colon F_n \to F_n}:{x \in X_n}$,
\item a strictly increasing sequence of integers $k_n \geq 2$,
\end{itemize}
such that for all $n \in \N$:\footnote{If the statement involves an object indexed by $n-1$ we only require that it holds for $n \geq 1$.}
\begin{enumerate}[label={\upshape($\dagger$\arabic{*})}]
	\item\label{nested} $G_{n-1} \subset G_{n}$ and $H_{n-1} \subset H_{n}$ as induced subgraphs,
    	\item\label{123} the vertices of $G_n$ and $H_n$ all have degree at most 5, 
        \item\label{12} the vertices of $F_n$ all have degree at most 3, 
       \item\label{colouredroots} the root of $G_n$ is in $R_n$ and the root of $H_n$ is in $B_n$,
	  \item\label{kbigenough} $\sigma_0(G_n)=\sigma_0(H_n) = k_n$,
        \item \label{specdis} $G_n$ and $H_n$ are spectrally distinguishable,
        \item \label{ends} $G_n$ and $H_n$ have at most one end,
        \item \label{dense} $\Omega(G_n \cup H_n) \subset \closure{R_n \cup B_n}$,
 		 \item \label{ballsG}
         		\begin{enumerate}[label=(\alph*)]
                	\item $G_{n}$ is a (proper) bare extension of infinite growth of $G_{n-1}$ at\newline $R_{n-1} \cup B_{n-1}$ to length $k_{n-1}+1$, and
                    \item $\Ball_{G_{n}}(G_{n-1}, k_{n-1}+1)$ does not meet $R_{n} \cup B_{n}$,
                \end{enumerate}
 		\item \label{ballsH}
        		\begin{enumerate}[label=(\alph*)]
                	\item $H_{n}$ is a (proper) bare extension of infinite growth of $H_{n-1}$ at\newline $R_{n-1} \cup B_{n-1}$ to length $k_{n-1}+1$, and
                    \item $\Ball_{H_{n}}(H_{n-1}, k_{n-1}+1)$ does not meet $R_{n} \cup B_{n}$,
                    \end{enumerate}
		\item \label{enumerations}there are enumerations $V(G_n) = \set{t_j}:{j \in J_n}$ and $V(H_n) = \set{s_j}:{j \in J_n}$ such that
		\begin{itemize}
    			\item $J_{n-1} \subset J_{n} \subset \N$, 
                \item $\set{t_j}:{j \in J_{n}}$ extends the enumeration $\set{t_j}:{j \in J_{n-1}}$ of $V(G_{n-1})$, and similarly for $\set{s_j}:{j \in J_n}$,
        			\item $\cardinality{\N \setminus J_n} = \infty$,
    			\item $\Set{0,1,\ldots,n} \subset J_n $,
       		\end{itemize}
 	\item \label{don'tmesswithleaves} $\set{t_j, s_j}:{j \leq n} \cap \p{R_n \cup B_n} = \emptyset$,
	\item\label{XandY} the finite sets of vertices $X_n$ and $Y_n$ satisfy $\cardinality{X_n} = n = \cardinality{Y_n}$, and
		\begin{itemize}
			\item $X_{n-1} \subset X_{n}$ and $Y_{n-1} \subset Y_{n}$,  
			\item $\varphi_{n} \restriction X_{n-1} = \varphi_{n-1}$,
			\item $\set{t_j}:{j \leq \lfloor (n-1)/2 \rfloor} \subset X_n$ and $\set{s_j}:{j \leq \lfloor n/2 \rfloor - 1} \subset Y_n$,
            \item $(X_n \cup Y_n) \cap (R_n \cup B_n) = \emptyset$,
		\end{itemize}
	\item\label{hypomorphism} the families of isomorphisms $\script{H}_n$ satisfy 
		\begin{itemize}
			\item  $h_{n, x} \restriction \p{G_{n-1} - x} = h_{n-1,x}$ for all $x \in X_{n-1}$,   
                         \item the image of $R_n \cap V(G_n)$ under $h_{n,x}$ is $R_n \cap V(H_n)$ for all $x \in X_n$,
                         \item the image of $B_n \cap V(G_n)$ under $h_{n,x}$ is $B_n \cap V(H_n)$ for all $x \in X_n$.
                         \end{itemize}
     \item\label{commutes} the families of automorphisms $\Pi_n$ satisfy 
		\begin{itemize}
                         \item $\pi_{n,x} \restriction R'_n$ is a permutation of $R'_n$ for each $x \in X_n$, 
                         \item $\pi_{n,x} \restriction B'_n$ is a permutation of $B'_n$ for each $x \in X_n$,   
     \item for each $x \in X_n$, the following diagram commutes: \begin{center}
\begin{tikzpicture}
  \matrix (m)
    [
      matrix of math nodes,
      row sep    = 3em,
      column sep = 9em
    ]
    {
      \script{L}(G_{n})              & \script{L}(H_{n}) \\
      \script{L}(F_{n}) &  \script{L}(F_{n})           \\
    };
  \path
    (m-1-1) edge [->] node [left] {\scriptsize{$\psi_{G_{n}}$}} (m-2-1)
    (m-1-1.east |- m-1-2)
      edge [->] node [above] {\scriptsize{$h_{n,x} \restriction \script{L}(G_{n})$}} (m-1-2);
        \path
     (m-2-1) edge [->] node [above] {\scriptsize{$\pi_{n,x}  \restriction \script{L}(F_{n})$}} (m-2-2);
        \path
     (m-1-2) edge [->] node [right] {\scriptsize{$\psi_{H_{n}}$}} (m-2-2);
\end{tikzpicture}
\end{center} 
\end{itemize}
	I.e.\ for every $\ell \in \script{L}(G_{n}) := V(G_n) \cap \p{R_n \cup B_n}$ we have $\pi_{n,x}(\psi_{G_n}(\ell)) = \psi_{H_n}(h_{n,x}(\ell))$.   
\end{enumerate}

\subsection{
The construction yields the desired non-reconstructible one-ended graphs.
}
\label{subsec:result}

By property~\ref{nested}, we have $G_0 \subset G_1 \subset G_2 \subset \cdots $ and $H_0 \subset H_1 \subset H_2 \subset \cdots$. Let $G$ and $H$ be the union of the respective sequences. Then both $G$ and $H$ are connected, and as a consequence of \ref{123}, both graphs have maximum degree $5$. 

We claim that the map $\varphi =\bigcup_{n} \varphi_n$ is a hypomorphism between $G$ and $H$. Indeed, it follows from \ref{enumerations} and \ref{XandY} that $\varphi$ is a well-defined bijection from $V(G)$ to $V(H)$. To see that $\varphi$ is a hypomorphism, consider any vertex $x$ of $G$. This vertex appears as some $t_j$ in our enumeration of $V(G)$, so the map
$$h_x = \bigcup_{n > 2j } h_{n,x} \colon G-x \to H-\varphi(x),$$ 
is a well-defined isomorphism by \ref{hypomorphism} between $G-x$ and $H-\varphi(x)$.

Now suppose for a contradiction that there exists an isomorphism $f \colon G \rightarrow H$. Then $f$ maps $t_0$ into $H_n$ for some $n \in \N$. Properties \ref{kbigenough} and \ref{ballsG} imply that after deleting all maximally bare paths in $G$ of length  $>k_n$, the connected component $C$ of $t_0$ is a leaf extension of $G_n$ adding one further leaf to every vertex in $V(G_n) \cap \p{R_n \cup B_n}$. Similarly, properties \ref{kbigenough} and \ref{ballsH} imply that after deleting all maximally bare paths in $H$ of length  $>k_n$, the connected component $D$ of $f(t_0)$ is a leaf-extension of $H_n$ adding one further leaf to every vertex in $V(H_n) \cap \p{R_n \cup B_n}$. Note that $f$ restricts to an isomorphism between $C$ and $D$. However, since $C$ and $D$ are proper extensions, we have $\Sigma(C) \triangle \Sigma(G_n) \subseteq \{1,2\}$ and $\Sigma(D) \triangle \Sigma(H_n) \subseteq \{1,2\}$. Hence, since $G_n$ and $H_n$ are spectrally distinguishable by \ref{specdis}, so are $C$ and $D$, a contradiction. We have established that $G$ and $H$ are non-isomorphic reconstructions of each other.

Finally, for $G$ being one-ended, we now show that for every finite vertex separator $S \subset V(G)$, the graph $G - S$ has only one infinite component (the argument for $H$ is similar). Suppose for a contradiction $G-S$ has two infinite components $C_1$ and $C_2$. Consider $n$ large enough such that $S \subset V(G_n)$. Since $G_k$ is one-ended for all $k$ by \ref{ends}, we may assume that $C_1 \cap G_k$ falls apart into finite components for all $k \geq n$. Since $C_1$ is infinite and connected, it follows from \ref{ballsG}(b) that $C_1$ intersects $G_{n+1} - G_n$. But since $G_{n+1}$ is an bare extension of $G_n$ of infinite growth by \ref{ballsG}(a), we see that that $C_1 \cap \p{G_{n+1} - G_n}$ contains an infinite component, a contradiction.

\subsection{
The base case: there are finite rooted graphs $G_0$ and $H_0$ satisfying requirements \ref{nested}--\ref{commutes}.
}
\label{subsec:basecase}

Choose a pair of spectrally distinguishable, equally sized graphs $G_0$ and $H_0$ with maximum degree $\leq 5$ and $\sigma_0(G_0) = \sigma_0(H_0) = k_0$. Pick a proper leaf each as roots $\rooot{G_0}$ and $\rooot{H_0}$ for $G_0$ and $H_0$, and further proper leaves $\ell_b \in G_0$ and $\ell_r \in H_0$.

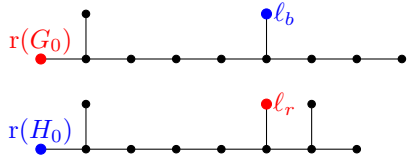
\begin{figure}[ht!]
\centering
\begin{tikzpicture}[scale=0.6]
\def \n {0}
\def \l {1}

\foreach \s in {1,...,9}
{
  \node[draw, circle,scale=.3, fill] (blob\s) at (\s,\n) {};
}

\foreach \s in {1,...,8}
{
	 \pgfmathsetmacro\t{\s + 1}
  \draw (blob\s) -- (blob\t);
}

\foreach \s in {2,6}
{
  \node[draw, circle,scale=.3, fill] (leaf\s) at (\s,\l) {};
   \draw (blob\s) -- (leaf\s);
}

\node[draw, red, circle,scale=.4, fill] (redroot) at (1,0) {};
\node[text=red] at (1,.4) {$\rooot{G_0}$};

\node[draw, blue, circle,scale=.4, fill] (blueleaf) at (6,1) {};
\node[text=blue] at (6.4,1) {$\ell_b$};

\def \n {-2}
\def \l {-1}

\foreach \s in {1,...,8}
{
  \node[draw, circle,scale=.3, fill] (blob\s) at (\s,\n) {};
}

\foreach \s in {1,...,7}
{
	 \pgfmathsetmacro\t{\s + 1}
  \draw (blob\s) -- (blob\t);
}

\foreach \s in {2,6,7}
{
  \node[draw, circle,scale=.3, fill] (leaf\s) at (\s,\l) {};
   \draw (blob\s) -- (leaf\s);
}

\node[draw, blue, circle,scale=.4, fill] (blueroot) at (1,-2) {};
\node[text=blue] at (1,-1.6) {$\rooot{H_0}$};

\node[draw, red, circle,scale=.4, fill] (redleaf) at (6,-1) {};
\node[text=red] at (6.4,-1) {$\ell_r$};

\end{tikzpicture}
\caption{A possible choice for the finite rooted graphs $G_0$ and $H_0$.}
\label{basecasefig}
\end{figure}

Define $R_0 = \{\rooot{G_0}, \ell_r \}$ and $B_0 = \{\rooot{H_0}, \ell_b\}$. We take $F_0$ to be two vertices $x$ and $y$ joined by an edge, with $R_0' = \{x\}$ and $B_0' = \{y\}$ and take $\psi_{G_0}$ to be the unique bijection sending $R_0 \cap G_0$ to $R'_0$ and $B_0 \cap G_0$ to $B'_0$, and similarly for $\psi_{H_0}$.
\begin{figure}[ht!]
\centering
\begin{tikzpicture}[scale=0.6]

\node[draw, red, circle,scale=.4, fill] (blob1) at (0,0) {};
\node[draw, blue, circle,scale=.4, fill] (blob2) at (1,0) {};

\draw (blob1) -- (blob2);

\node[text=blue] at (1.5,0) {$y$};

\node[text=red] at (-0.5,0) {$x$};

\end{tikzpicture}
\caption{$F_0$.}
\label{basecaseF}
\end{figure}

Let $J_0 = \Set{0,1,\ldots, |G_0|-1}$ and choose enumerations $V(G_0) = \set{t_j}:{j \in J_0}$ and $V(H_0) = \set{s_j}:{j \in J_0}$ with $t_0 \neq \rooot{G_0}$ and $s_0 \neq \rooot{H_0}$. Finally we let $X_0 = Y_0 = \script{H}_0 = \emptyset$. It is a simple check that conditions \ref{nested}--\ref{commutes} are satisfied.

\subsection{The inductive step: set-up}
\label{subsec:setup}

Now, assume that we have constructed graphs $G_k$ and $H_k$ for all $k \leq n$ such that \ref{nested}--\ref{commutes} are satisfied up to $n$. If $n=2m$ is even, then we have $\set{t_j}:{j \leq m-1} \subset X_{n}$ and in order to satisfy \ref{XandY} we have to construct $G_{n+1}$ and $H_{n+1}$ such that the vertex $t_m$ is taken care of in our partial hypomorphism. Similarly, if $n=2m+1$ is odd, then we have $\set{s_j}:{j \leq m-1} \subset Y_{n}$ and we have to construct $G_{n+1}$ and $H_{n+1}$ such that the vertex $s_m$ is taken care of in our partial hypomorphism. Both cases are symmetric, so let us assume in the following that $n=2m$ is even.

Now let $v$ be the vertex with the least index in the set $\set{t_j}:{j \in J_n} \setminus X_n$, i.e.\
\begin{align}
 v = t_i \; \text{ for } \; i = \min \set{j}:{t_j \in V(G_n) \setminus X_n}. \numberthis\label{defofx}
\end{align}

Then by assumption~\ref{XandY}, $v$ will be $t_m$, unless $t_m$ was already in $X_n$ anyway. In any case, since $\cardinality{X_n}= \cardinality{Y_n} = n$, it follows from \ref{enumerations} that $i \leq n$, so by \ref{don'tmesswithleaves}, $v$ does not lie in our leaf sets $R_n \cup B_n$, i.e.\
\begin{align}
 v \notin R_n \cup B_n. \numberthis\label{xdoesntmess}
\end{align}

In the next sections, we will demonstrate how to obtain graphs $G_{n+1} \supset G_{n}$, $H_{n+1}\supset H_{n}$ and $F_{n+1}$ with $X_{n+1}=X_n \cup \singleton{v}$ and $Y_{n+1} = Y_n \cup \singleton{\varphi_{n+1}(v)}$ satisfying \ref{nested}---\ref{ballsH} and \ref{XandY}--\ref{commutes}.

After we have completed this step, since $\cardinality{\N \setminus J_{n}} = \infty$, it is clear that we can extend our enumerations of $G_{n}$ and $H_{n}$ to enumerations of $G_{n+1}$ and $H_{n+1}$ as required, making sure to first list some new elements that do not lie in $R_{n+1} \cup B_{n+1}$.  This takes care of \ref{enumerations} and \ref{don'tmesswithleaves} and completes the step $n \mapsto n+1$.

\subsection{The inductive step: construction}
\label{subsec:constrinducstep}

We will construct the graphs $G_{n+1}$ and $H_{n+1}$ in three steps. First, in \S\ref{subsubsec:building} we construct graphs $G'_{n+1} \supset G_n$ and $H'_{n+1} \supset H_n$ such that there is a vertex $\phi_{n+1}(v) \in H'_{n+1}$ with $G'_{n+1} - v \cong H'_{n+1} - \phi_{n+1}(v)$. This first step essentially follows the argument from \cite[\S4.6]{BEHLP17}. We will also construct a graph $F_{n+1}$ via a parallel process. 

Secondly, in \S\ref{subsubsec:extending} we will show that there are well-behaved maps from the coloured leaves of $G'_{n+1}$ and $H'_{n+1}$ to $F_{n+1} \times \mathbb{N}$, such that analogues of \ref{hypomorphism} and \ref{commutes} hold for $G'_{n+1}$, $H'_{n+1}$ and $F_{n+1}$, giving us control over the corresponding gluing sum. 

Lastly, in \S\ref{subsubsec:gluing}, we do the actual gluing process and define all objects needed for step $n+1$ of our inductive construction.

\subsubsection{Building the auxiliary graphs}
\label{subsubsec:building}

Given the two graphs $G_n$ and $H_n$, we extend each of them through their roots as indicated in Figure~\ref{constructionfigure2} to graphs $\tilde{G}_n$ and $\tilde{H}_n$ respectively. 

Since $v$ is not the root of $G_n$, there is a unique component of $G_n - v$ containing the root, which we call $G_n(r)$. Let $G_n(v)$ be the induced subgraph of $G_n$ on the remaining vertices, including $v$. We remark that if $v$ is not a cutvertex of $G_n$, then $G_n(v)$ is just a single vertex $v$. Since $\sigma_0(G_n) = k_n$ by \ref{kbigenough} and $\deg(v) \leq 5$ by \ref{123}, it follows from an iterative application of Lemma~\ref{l:miibound} that $\Sigma\left(G_n(r)\right)$ and $\Sigma\left(G_n(v)\right)$ are finite. Let $k=\tilde{k}_n = \max \{ \sigma_0(G_n), \sigma_0 \left(G_n(r)\right), \sigma_0\left(G_n(v)\right), \sigma_0(H_n) \} + 1$.

\begin{figure}[ht!]
\begin{subfigure}[t]{0.5\textwidth}
\centering
\begin{tikzpicture}[scale=0.6]

\node[draw, circle,scale=.3, fill] (Ttop) at (1,3) {};
\node at (1.4,3.5) {$\rooot{G_{n}}$};
\node[draw, circle,scale=.3, fill] (extra1) at (0,3) {};
\draw (Ttop) -- (extra1);

\node[draw, circle,scale=.3, fill] (Stop) at (7,3) {};
\node[draw, circle,scale=.3, fill] (extra2) at (8,3) {};
\draw (Stop) -- (extra2);
  
\draw (0,0) -- (2,0) -- (1,2) -- (0,0);
\draw[->,red,very thick] (1,2) -- (1,3);

\node at (1.5,-.6) {$G_n$};

\node[draw, circle,scale=.3, fill] (vtx) at (1.05,.4) {};
\node at (1.5,.4) {$v$};

\draw (6,0) -- (8,0) -- (7,2) -- (6,0);
\draw (7,2) -- (7,3);

\node at (8.0,-0.6) {$\hat{H}_n$};

\node[draw, blue, circle,scale=.3, fill] () at (.5,.0) {};
\node[draw, blue, circle,scale=.3, fill] () at (1.2,.0){};
\node[draw, blue, circle,scale=.3, fill] () at (1.7,.0)  {};
\node[draw, red, circle,scale=.3, fill] () at (.7,.0) {};
\node[draw, red, circle,scale=.3, fill] () at (1.4,.0) {};

\node[draw, red, circle,scale=.3, fill] () at (6.5,.0) {};
\node[draw, red, circle,scale=.3, fill] () at (7.2,.0) {};
\node[draw, red, circle,scale=.3, fill] () at (7.7,.0)  {};
\node[draw, blue, circle,scale=.3, fill] () at  (6.7,.0) {};
\node[draw, blue, circle,scale=.3, fill] () at (7.4,.0) {};
  
\node[draw, circle,scale=.3, fill] (gadget1) at (3.5,3) {};
\node[draw, circle,scale=.3, fill] (gadget2) at (4.5,3) {};
\draw (gadget1) -- (gadget2);
    
\draw[dashed] (Ttop) -- (gadget1);
\draw[dashed] (Stop) -- (gadget2);
    
\draw[->,yellow,very thick] (gadget1) -- (3.5,4);
\node[draw, yellow, circle,scale=.4, fill] (leaf1) at (3.5,4) {};
\node at (2.7,4.5) {$\rooot{G'_{n+1}}$};
\node[draw, green, circle,scale=.3, fill] (leaf1) at (4.5,4) {};
\draw (gadget2) -- (leaf1);
\node at (4.5,4.5) {$g$};
     
\end{tikzpicture}
\caption*{The graph $\tilde{G}_{n}$.}
\end{subfigure}%
\begin{subfigure}[t]{0.5\textwidth}
\centering
\begin{tikzpicture}[scale=0.6]

\node[draw, circle,scale=.3, fill] (Ttop) at (1,3) {};
\node[draw, circle,scale=.3, fill] (extra1) at (0,3) {};
\draw (Ttop) -- (extra1);
\node[draw, circle,scale=.3, fill] (Stop) at (7,3) {};
\node at (6.8,3.5) {$\rooot{H_{n}}$};
\node[draw, circle,scale=.3, fill] (extra2) at (8,3) {};
\draw (Stop) -- (extra2);
  
\draw (0,0) -- (1,0) -- (1.6,1) -- (1,2) -- (0,0);
\draw (1,2) -- (1,3);
\node at (1,-.6) {$\hat{G}_n(r)$};

\node at (7.5,.6) {};

\draw (2,1) -- (3,1) -- (2.5,2) -- (2,1);
\node[draw, circle,scale=.3, fill] (blob3) at (2.5,2) {};
\node[draw, circle,scale=.3, fill] (blob4) at (2.5,3) {};
\draw (blob3) -- (blob4);
\node at (3,2) {$\hat{v}$};
\node[] at (2.6,.5) {$\hat{G}_n(\hat{v})$};
\node[draw, blue, circle,scale=.3, fill] () at (2.2,1){};
\node[draw, red, circle,scale=.3, fill] () at (2.5,1)  {};
\node[draw, blue, circle,scale=.3, fill] () at (2.8,1) {};

\draw (6,0) -- (8,0) -- (7,2) -- (6,0);
\draw[->,blue,very thick] (7,2) -- (7,3);

\node at (8.0,-.6) {$H_n$};

\node[draw, blue, circle,scale=.3, fill] () at (.5,.0) {};
\node[draw, red, circle,scale=.3, fill] () at (.7,.0) {};

\node[draw, red, circle,scale=.3, fill] () at (6.5,.0) {};
\node[draw, red, circle,scale=.3, fill] () at (7.2,.0) {};
\node[draw, red, circle,scale=.3, fill] () at (7.7,.0)  {};
\node[draw, blue, circle,scale=.3, fill] () at  (6.7,.0) {};
\node[draw, blue, circle,scale=.3, fill] () at (7.4,.0) {};
  
\node[draw, circle,scale=.3, fill] (gadget1) at (3.5,3) {};
\node[draw, circle,scale=.3, fill] (gadget2) at (4.5,3) {};
\draw (gadget1) -- (gadget2);
    
\draw[dotted](Ttop) -- (gadget1);
\draw[dotted] (Stop) -- (gadget2);
    
\draw[->,green,very thick] (gadget2) -- (4.5,4);
        
\node[draw, green, circle,scale=.3, fill] (leaf1) at (4.5,4) {};
\node at (5.4,4.5) {$\rooot{H'_{n+1}}$};

\node[draw, yellow, circle,scale=.3, fill] (leaf1) at (3.5,4) {};
\node at (3.5,4.5) {$y$};

\draw (gadget1) -- (leaf1);
     
\end{tikzpicture}
\caption*{The graph $\tilde{H}_{n}$.}
\end{subfigure}
\caption{All dotted lines are maximally bare paths of length at least $k+1=\tilde{k}_n+1$.}
\label{constructionfigure2}
\end{figure}
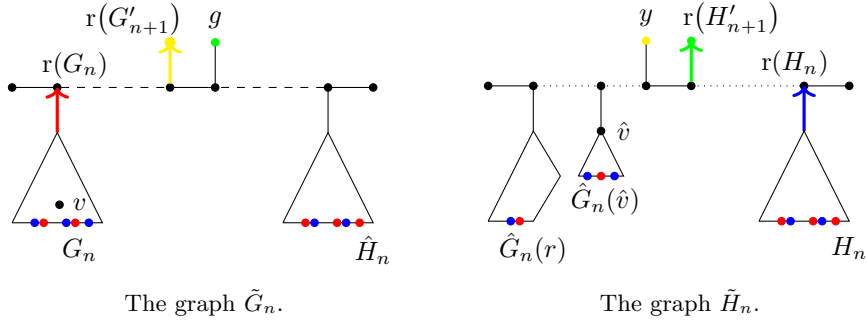

To obtain $\tilde{G}_n$, we extend $G_n$ through its root $\rooot{G_n} \in R_n$ by a path 
$$\rooot{G_n} = u_0, u_1,\dotsc,u_{p-1},u_{p}=\rooot{\hat{H}_n}$$ 
of length $p=4(\tilde{k}_n+1) + 1$, where at its last vertex $u_p$ we glue a rooted copy $\hat{H}_n$ of $H_n$ (via an isomorphism $\hat{z} \leftrightarrow z$), identifying $u_p$ with the root of $\hat{H}_n$.

Next, we add two additional leaves at $u_0$ and $u_p$, so that $\deg(\rooot{G_n})=3=\deg\left(\rooot{\hat{H}_n}\right)$. Further, we add a leaf $\rooot{G'_{n+1}}$ at $u_{2k+2}$, which will be our new root for the next tree $G'_{n+1}$; and another leaf $g$ at $u_{2k+3}$. This completes the construction of $\tilde{G}_n$.

The construction of $\tilde{H}_n$ is similar, but not entirely symmetric. For its construction, we extend $H_n$ through its root $\rooot{H_n} \in B_n$ by a path
$$\rooot{H_n} = v_p, v_{p-1}, \dotsc,v_1,v_0=\rooot{\hat{G}_n(r)}$$ 
of length $p$, where at its last vertex $v_0$ we glue a copy $\hat{G}_n(r)$ of $G_n(r)$, identifying $v_0$ with the root of $\hat{G}_n(r)$. Then, we take a copy $\hat{G}_n(\hat{v})$ of $G_n(v)$ and connect $\hat{v}$ via an edge to $v_{k+1}$.

Finally, as before, we add two leaves at $v_0$ and $v_p$ so that $\deg\left(\rooot{\hat{G}_n(r)}\right)=3=\deg\left(\rooot{H_n}\right)$. Next, we add a leaf $\rooot{H'_{n+1}}$ to $v_{2k+3}$, which will be our new root for the next tree $H'_{n+1}$; and another leaf $y$ to $v_{2k+2}$. This completes the construction of $\tilde{H}_n$.

By the induction assumption, certain leaves of $G_n$ have been coloured with one of the two colours in $R_n \cup B_n$, and also some leaves of $H_n$ have been coloured with one of the two colours in $R_n \cup B_n$. In the above construction, we colour leaves of $\hat{H}_n$, $\hat{G}_n(r)$ and $\hat{G}_n(\hat{v})$ accordingly:
\begin{align}
\begin{split}
\tilde{R}_n &= \p{R_n \cup \set{\hat{z} \in \hat{H}_n \cup \hat{G}_n(r) \cup \hat{G}_n(\hat{v})}:{z \in R_n}} \setminus \Set{\rooot{G_n},\rooot{\hat{G}_n(r)}},  \nonumber \\
\tilde{B}_n &= \p{B_n \cup \set{\hat{z} \in \hat{H}_n \cup \hat{G}_n(r) \cup \hat{G}_n(\hat{v})}:{z \in B_n}} \setminus \Set{\rooot{H_n},\rooot{\hat{H}_n}}.
\end{split}
\numberthis\label{defoftildeBn}
\end{align}

Now put $M_n := \tilde{G}_n \cup \tilde{H}_n$ and consider the following promise structure $\script{P}=\p{M_n, \vec{P}, \script{L}}$ on $M_n$, consisting of four promise edges $\vec{P} = \Set{\vec{p}_1,\vec{p}_2,\vec{p}_3, \vec{p}_4}$ and corresponding leaf sets $\script{L}=\Set{L_1, L_2, L_3, L_4}$, as follows:
\begin{align} 
\begin{split}
\bullet \; \vec{p}_1 &\text{ pointing in } G_{n} \text{ towards }\rooot{G_n}, \text{ with } L_1=\tilde{R}_n, \\
\bullet \; \vec{p}_2 &\text{ pointing in } H_n \text{ towards }\rooot{H_n}, \text{ with } L_2=\tilde{B}_n, \\
\bullet \; \vec{p}_3 &\text{ pointing in }\tilde{G}_n \text{ towards }\rooot{G'_{n+1}}, \text{ with } L_3=\Set{\rooot{G'_{n+1}},y}, \\
\bullet \; \vec{p}_4 &\text{ pointing in } \tilde{H}_n \text{ towards } \rooot{H'_{n+1}}, \text{ with } L_4=\Set{\rooot{H'_{n+1}}, g}.
\end{split}
\numberthis\label{defofpromises}
\end{align}

Note that our construction so far has been tailored to provide us with a $\vec{P}$-respecting isomorphism
\begin{align}
h \colon \tilde{G}_n - v \to \tilde{H}_n - \hat{v}. \numberthis \label{defofhn+1}
\end{align}

Consider the closure $\cl(M_n)$ with respect to the above defined promise structure $\script{P}$. 
Since $\cl(M_n)$ is a leaf-extension of $M_n$, it has two connected components, just as $M_n$. We now define
\begin{align}
\begin{split}
G'_{n+1} &= \text{ the component containing } G_n \text{ in } \cl(M_n), \\
H'_{n+1} &= \text{ the component containing } H_n \text{ in } \cl(M_n).
\end{split}
\numberthis\label{defofTnplus1andSnplus1}
\end{align}
It follows that $\cl(M_n) = G'_{n+1} \cup H'_{n+1}$. Further, since $\vec{p}_3$ and $\vec{p}_4$ are placeholder promises, $\cl(M_n)$ carries a corresponding promise structure, cf. Def.~\ref{def_closurestructure}. We define 
\begin{align}
R_{n+1} = \cl(L_3) \; \text{ and } \; B_{n+1} = \cl(L_4).
\numberthis\label{defofKnplus1andLnplus1}
\end{align}

Lastly, set
\begin{align}
\begin{split}
X_{n+1} &= X_n \cup \singleton{v}, \\
Y_{n+1} &= Y_n \cup \singleton{\hat{v}}, \\
\varphi_{n+1} &= \varphi_n \cup \Set{(v,\hat{v})}, \\
k_{n+1} &= 2(\tilde{k}_n + 1).
\end{split}\numberthis \label{nextXandY}
\end{align}

We now build $F_{n+1}$ in a similar fashion to the above procedure. That is, we take two copies of $F_n$ and join them pairwise through their roots as indicated in Figure~\ref{sketchfigconstrF} to form a graph $\tilde{F}_n$. We consider the graph $N_n = \tilde{F}_n \cup \hat{\tilde{F}}_n$, and take $F_{n+1}$ to be one of the components of $\cl(N_n)$ (unlike for $\cl(M_n)$, both components of $\cl(N_n)$ are isomorphic).

\begin{figure}[ht!]
\begin{subfigure}[t]{0.5\textwidth}
\centering
\begin{tikzpicture}[scale=0.6]

\node[draw, circle,scale=.3, fill] (Ttop) at (1,3) {};
\node at (1,3.5) {$\psi_{G_n}(\rooot{G_n})$};
\node at (7,3.5) {$\psi_{H_n}(\rooot{H_n})$};

\node[draw, circle,scale=.3, fill] (Stop) at (7,3) {};
  
\draw (0,0) -- (2,0) -- (2,2) -- (0,2) -- (0,0);
\draw[->,red,very thick] (1,2) -- (1,3);

\node at (1,1) {$F^G_n$};
  
\draw (6,0) -- (8,0) -- (8,2) -- (6,2) -- (6,0);
\draw (7,2) -- (7,3);

\node at (7,1) {$F^H_n$};

\node[draw, red, circle,scale=.3, fill] () at (1,.0) {};
\node[draw, blue, circle,scale=.3, fill] () at (0.1,.0) {};
\node[draw, blue, circle,scale=.3, fill] () at (1.9,.0) {};

\node[draw, red, circle,scale=.3, fill] () at (6.1,.0) {};
\node[draw, red, circle,scale=.3, fill] () at  (7,.0) {};
\node[draw, blue, circle,scale=.3, fill] () at (7.9,.0) {};
  
\node[draw, circle,scale=.3, fill] (gadget1) at (3.5,3) {};
\node[draw, circle,scale=.3, fill] (gadget2) at (4.5,3) {};
\draw (gadget1) -- (gadget2);
    
\draw (Ttop) -- (gadget1);

\draw (Stop) -- (gadget2);

\draw[->,yellow,very thick](gadget1) -- (3.5,4);
\node[draw, yellow, circle,scale=.4, fill] (leaf1) at (3.5,4) {};
\node[draw, green, circle,scale=.4, fill] (leaf1) at (4.5,4) {};
\draw (gadget2) -- (leaf1);
\node at (3.5,4.5) {$x$};
\node at (4.5,4.5) {$y$};
     
\end{tikzpicture}
\subcaption*{The graph $\tilde{F_n}$.}

\end{subfigure}%
\begin{subfigure}[t]{0.5\textwidth}
\centering
\begin{tikzpicture}[scale=0.6]

\node[draw, circle,scale=.3, fill] (Ttop) at (1,3) {};
\node at (1,3.5) {$\widehat{\psi_{G_n}(\rooot{G_n})}$};
\node at (7,3.5) {$\widehat{\psi_{H_n}(\rooot{H_n})}$};

\node[draw, circle,scale=.3, fill] (Stop) at (7,3) {};
  
\draw (0,0) -- (2,0) -- (2,2) -- (0,2) -- (0,0);
\draw (1,2) -- (1,3);

\node at (1.2,1) {$\hat{F}^G_n$};
  
\draw (6,0) -- (8,0) -- (8,2) -- (6,2) -- (6,0);
\draw[->,blue,very thick] (7,2) -- (7,3);

\node at (7,1) {$\hat{F}^H_n$};

\node[draw, red, circle,scale=.3, fill] () at (1,.0) {};
\node[draw, blue, circle,scale=.3, fill] () at (0.1,.0) {};
\node[draw, blue, circle,scale=.3, fill] () at (1.9,.0) {};

\node[draw, red, circle,scale=.3, fill] () at (6.1,.0) {};
\node[draw, red, circle,scale=.3, fill] () at  (7,.0) {};
\node[draw, blue, circle,scale=.3, fill] () at (7.9,.0) {};
  
\node[draw, circle,scale=.3, fill] (gadget1) at (3.5,3) {};
\node[draw, circle,scale=.3, fill] (gadget2) at (4.5,3) {};
\draw (gadget1) -- (gadget2);
    
\draw (Ttop) -- (gadget1);
\draw (Stop) -- (gadget2);
    
\draw(gadget1) -- (3.5,4);
\node[draw, yellow, circle,scale=.4, fill] (leaf1) at (3.5,4) {};
\node[draw, green, circle,scale=.4, fill] (leaf1) at (4.5,4) {};
\draw[->,green,very thick] (gadget2) -- (leaf1);
\node at (3.5,4.5) {$\hat{x}$};
\node at (4.5,4.5) {$\hat{y}$};

\end{tikzpicture}
\subcaption*{The graph $\hat{\tilde{F_n}}$.}
\end{subfigure}
\caption{The graph $N_n = \tilde{F}_n \cup \hat{\tilde{F}}_n$.}
\label{sketchfigconstrF}
\end{figure}
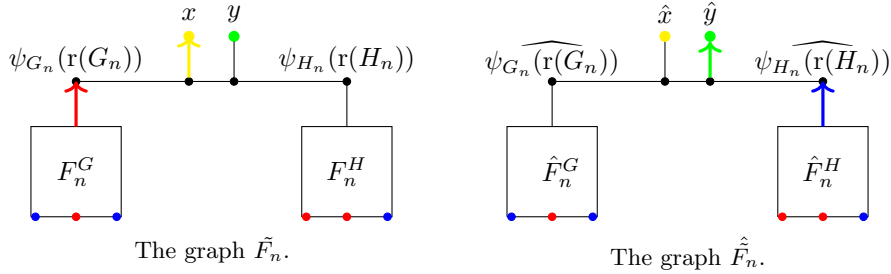

More precisely we take two copies of $F_n$, which we will denote by $F^G_n$ and $F^H_n$. We extend $F^G_n$ through the image of the $\rooot{G_n}$ under the bijection $\psi_{G_n}$ by a path
$$\psi_{G_n}(\rooot{G_n}) = u_0, u_1,u_2,u_3 = \psi_{H_n}(\rooot{H_n})$$ 
of length three, where $\psi_{G_n}(\rooot{G_n})$ is taken in $F^G_n$ and $\psi_{H_n}(\rooot{H_n})$ is taken in $F^H_n$. Further, we add a leaf $x$ at $u_1$, and another leaf $y$ at $u_2$. We will consider the graph $N_n = \tilde{F}_n \cup \hat{\tilde{F}}_n$ as in Figure \ref{sketchfigconstrF} formed by taking two disjoint copies of $\tilde{F}_n$.

By the induction assumption, certain leaves of $F_n$ have been coloured with one of the two colours in $R'_n \cup B'_n$. In the above construction, we colour leaves of $F^G_n,F^H_n,\hat{F}^G_n$ and $\hat{F}^H_n$ accordingly:
\begin{align}
\begin{split}
\tilde{R}'_n &= \Set{w \in F^G_n \cup F^H_n \cup\hat{F}^G_n \cup \hat{F}^H_n\, : \, w \in R'_n } \setminus \Set{\psi_{G_n}(\rooot{G_n}) , \widehat{\psi_{G_n}(\rooot{G_n})}} \nonumber \\
\tilde{B}'_n &= \Set{w \in F^G_n \cup F^H_n \cup\hat{F}^G_n \cup \hat{F}^H_n\, : \, w \in B'_n } \setminus \Set{\psi_{H_n}(\rooot{H_n}), \widehat{\psi_{H_n}(\rooot{H_n})}}.
\end{split}
\numberthis\label{defofBn'}
\end{align}

Now consider the following promise structure $\script{P}'=\p{N_n, \vec{P}', \script{L}'}$ on $N_n$, consisting of four promise edges $\vec{P}' = \Set{\vec{r}_1,\vec{r}_2,\vec{r}_3, \vec{r}_4}$ and corresponding leaf sets $\script{L}'=\Set{L'_1, L'_2, L'_3, L'_4}$, as follows:
\begin{align} 
\begin{split}
\bullet \; \vec{r}_1 &\text{ pointing in } F^G_n \text{ towards } \psi_{G_n}(\rooot{G_n}), \text{ with } L'_1=\tilde{R}'_n, \\
\bullet \; \vec{r}_2 &\text{ pointing in } \hat{F}^H_n \text{ towards } \widehat{\psi_{H_n}(\rooot{H_n})}, \text{ with } L'_2=\tilde{B}'_n, \\
\bullet \; \vec{r}_3 &\text{ pointing in }\tilde{F}_n \text{ towards } x, \text{ with } L'_3=\Set{x,\hat{x}}, \\
\bullet \; \vec{r}_4 &\text{ pointing in }\hat{\tilde{F}}_n \text{ towards } \hat{y}, \text{ with } L'_4=\Set{y,\hat{y}}.
\end{split}
\numberthis\label{defofpromises'}
\end{align}

Consider the closure $\cl(N_n)$ with respect to the promise structure $\script{P}'$ defined above. 
Since $\cl(N_n)$ is a leaf-extension of $N_n$, it has two connected components, and we define $F_{n+1}$ to be the component containing $F^G_n$ in $\cl(N_n)$. Since $\vec{r}_3$ and $\vec{r}_4$ are placeholder promises, $\cl(N_n)$ carries a corresponding promise structure, cf. Def.~\ref{def_closurestructure}. We define 
\begin{align}
R'_{n+1} = \cl(L'_3) \cap F_{n+1} \; \text{ and } \; B'_{n+1} = \cl(L'_4) \cap F_{n+1} .
\numberthis\label{defofnewFpromise}
\end{align}

\subsubsection{Extending maps}
\label{subsubsec:extending}

In order to glue $F_{n+1} \times \mathbb{N}$ onto $G'_{n+1}$ and $H'_{n+1}$ we will need to show that that analogues of \ref{hypomorphism} and \ref{commutes} hold for $G'_{n+1}$, $H'_{n+1}$ and $F_{n+1}$. Our next lemma is essentially \cite[Claim 4.13]{BEHLP17}, and is an analogue of \ref{hypomorphism}. We briefly remind the reader of the details, as we need to know the nature of our extensions in our later claims. 

\begin{lemma}\label{petersclaim}
There is a family of isomorphisms $\script{H}'_{n+1} = \set{h'_{n+1,x}}:{x \in X_{n+1}}$ witnessing that $G'_{n+1} - x$ and $H'_{n+1} - \varphi_{n+1}(x)$ are isomorphic for all $x \in X_{n+1}$, such that $h'_{n+1,x}$ extends $h_{n,x}$ for all $x \in X_n$. 
\end{lemma}
\begin{proof}
The graphs $G'_{n+1}$ and $H'_{n+1}$ defined in (\ref{defofTnplus1andSnplus1}) are obtained from $\tilde{G}_n$ and $\tilde{H}_n$ by attaching at every leaf in $\tilde{R}_n$ a copy of the rooted graph $\cl(M_n)(\vec{p}_1)$, and by attaching at every leaf in $\tilde{B}_n$ a copy of the rooted graph $\cl(M_n)(\vec{p}_2)$ by \ref{i:keepspromises}.

From (\ref{defofhn+1}) we know that there is a $\vec{P}$-respecting isomorphism 
$$h \colon \tilde{G}_n - v \to \tilde{H}_n - \varphi_{n+1}(v).$$
In other words, $h$ maps promise leaves in $L_i \cap V(\tilde{G}_n)$ bijectively to the promise leaves in $L_i \cap V(\tilde{H}_n)$ for all  $i=1,2,3,4$.

There is for each $\ell \in \tilde{R}_n \cup \tilde{B}_n \cup \{ \rooot{G_n}, \rooot{H_n}\}$ a $\cl(\vec{P})$-respecting isomorphism of rooted graphs
\begin{align}
f_\ell \colon  \cl(M_n)(\vec{q}_{\ell}) \cong \cl(M_n)(\vec{p}_i) \numberthis \label{f_ell}
\end{align}
given by \ref{i:keepslabelledpromises} for $\ell \in (\tilde{R}_n \cup \tilde{B}_n)$, where $i$ equals blue or red depending on whether $\ell \in \tilde{R}_n$ or $\tilde{B}_n$, and for the roots of $G_n$ and $H_n$ we have $\vec{q_r}= \vec{p_i}$ and the isomorphism is the identity. Hence, for each $\ell$,
\[
f_{h(\ell)}^{-1} \circ f_\ell \colon \cl(M_n)(\vec{q}_{\ell}) \cong \cl(M_n)(\vec{q}_{h(\ell)})
\]
is a $\cl(\vec{P})$-respecting isomorphism of rooted graphs. By combining the isomorphism $h$ between $\tilde{G}_n - v$ and $\tilde{H}_n - \varphi_{n+1}(v)$ with these isomorphisms between each $\cl(M_n)(\vec{q}_{\ell})$ and $\cl(M_n)(\vec{q}_{h(\ell)})$ we get a $\cl(\vec{P})$-respecting isomorphism 
$$h'_{n+1,v} \colon G'_{n+1} - v \to H'_{n+1} - \varphi_{n+1}(v).$$

To extend the old isomorphisms $h_{n,x}$ (for $x \in X_n$), note that $G'_{n+1}$ and $H'_{n+1}$ are obtained from $G_n$ and $H_n$ by attaching at every leaf in $R_n$ a copy of the rooted graph $\cl(M_n)(\vec{p}_1)$, and similarly by attaching at every leaf in $B_n$ a copy of the rooted graph $\cl(M_n)(\vec{p}_2)$. By induction assumption \ref{hypomorphism}, for each $x \in X_n$ the isomorphism
$$ h_{n,x} \colon G_n - x \rightarrow H_n - \varphi_n(x) $$
maps the red leaves of $G_n$ bijectively to the red leaves of $H_n$, and the blue leaves of $G_n$ bijectively to the blue leaves of $H_n$. Thus, by (\ref{f_ell}),
$$f_{h_{n,x}(\ell)}^{-1} \circ f_\ell \colon \cl(M_n)(\vec{q}_\ell) \cong \cl(M_n)(\vec{q}_{h_{n,x}(\ell)})$$
are $\cl(\vec{P})$-respecting isomorphisms of rooted graphs for all $\ell \in (R_n \cup B_n) \cap V(G_n)$. By combining the isomorphism $h_{n,x}$ between $G_n - x$ and $H_n - \varphi_n(x)$ with these isomorphisms between each $\cl(M_n)(\vec{q}_{\ell})=G'_{n+1}(\vec{q}_{\ell})$ and $\cl(M_n)(\vec{q}_{h_{n,x}(l)})=H'_{n+1}(\vec{q}_{h_{n,x}(l)})$, we obtain a $\cl(\vec{P})$-respecting extension 
\[
h'_{n+1,x} \colon G'_{n+1} - x \rightarrow H'_{n+1} - \varphi_n(x). \qedhere
\]
\end{proof}

Our next claim should be seen as an approximation to property \ref{commutes}. Recall that $\cl(N_n)$ has two components $F_{n+1} \cong \hat{F}_{n+1}$.
\begin{lemma}\label{maxclaim}
There are colour-preserving bijections 
$$\psi_{G'_{n+1}} \colon V(G'_{n+1}) \cap \p{R_{n+1} \cup B_{n+1}} \to R'_{n+1} \cup B'_{n+1},$$  
$$\psi_{H'_{n+1}} \colon V(H'_{n+1}) \cap \p{R_{n+1} \cup B_{n+1}} \to 
\hat{R}'_{n+1} \cup \hat{B}'_{n+1},$$
and a family of isomorphisms 
$$\hat{\Pi}_{n+1} = \set{\hat{\pi}_{n+1,x} \colon F_{n+1} \to \hat{F}_{n+1}}:{x \in X_{n+1}}$$
such that for each $x \in X_{n+1}$ the following diagram commutes.
\begin{center}
\begin{tikzpicture}
  \matrix (m)
    [
      matrix of math nodes,
      row sep    = 3em,
      column sep = 9em
    ]
    {
      \script{L}(G'_{n+1})              & \script{L}(H'_{n+1}) \\
      \script{L}'(F_{n+1}) &  \script{L}'(\hat{F}_{n+1})           \\
    };
  \path
    (m-1-1) edge [->] node [left] {\scriptsize{$\psi_{G'_{n+1}}$}} (m-2-1)
    (m-1-1.east |- m-1-2)
      edge [->] node [above] {\scriptsize{$h'_{n+1,x} \restriction \script{L}(G'_{n+1})$}} (m-1-2);
        \path
     (m-2-1) edge [->] node [above] {\scriptsize{$\hat{\pi}_{n+1,x}  \restriction \script{L}'(F_{n+1})$}} (m-2-2);
        \path
     (m-1-2) edge [->] node [right] {\scriptsize{$\psi_{H'_{n+1}}$}} (m-2-2);
\end{tikzpicture}
\end{center}
\end{lemma}

\begin{proof} 
{\em Defining $\psi_{G'_{n+1}}$ and $\psi_{H'_{n+1}}$.} By construction, we can combine the maps $\psi_{G_n}$ and $\psi_{H_n}$ to obtain a natural colour-preserving bijection 
$$\psi \colon \script{L}\p{M_n} \to \script{L}'\p{N_n},$$ 
which satisfies the assumptions of Lemma~\ref{l:blabla}. Thus, by Corollary~\ref{cor_blablablabla}, there are bijections 
$$\alpha^i \colon \script{L}\p{\cl(M_n)(\vec{p}_i)} \to \script{L}'\p{\cl(N_n)(\vec{r}_i)}$$ 
which are colour-preserving with respect to the promise structures $\cl(\script{P})$ and $\cl(\script{P}')$ on $\cl(M_n)$ and $\cl(N_n)$, respectively.

We now claim that $\psi$ extends to a colour-preserving bijection (w.r.t.\ $\cl(\script{P})$)
$$\cl(\psi) \colon \script{L}\p{\cl(M_n)} \to \script{L}'\p{\cl(N_n)}.$$
Indeed, by \ref{i:keepslabelledpromises}, for every $\ell \in \tilde{R}'_n \cup \tilde{B}'_n$, there is a $\vec{P'}$-respecting rooted isomorphism
\begin{align}
g_\ell \colon \cl(N_n)(\vec{q}_\ell) \to \cl(N_n)(\vec{r}_i), \numberthis \label{g_ell}
\end{align}
where $i$ equals blue or red depending on whether $\ell \in \tilde{R}'_n$ or $\tilde{B}'_n$. As in the case of (\ref{f_ell}) we define the maps $g_{r}$ with $\vec{q}_r = \vec{r}_i$ for the roots of $F^G_n$ and $\hat{F}^H_n$ respectively to be the identity.
Together with the rooted isomorphisms $f_\ell$ from (\ref{f_ell}), it follows that for each $\ell \in  \tilde{R}_n \cup \tilde{B}_n \cup \{ \rooot{G_n}, \rooot{H_n}\}$, the map
$$\psi_\ell = g_{\psi(\ell)}^{-1} \circ \alpha^i \circ f_\ell \colon \script{L}\p{\cl(M_n)(\vec{q}_\ell)} \to \script{L}\p{\cl(N_n)(\vec{q}_{\psi(\ell)})}$$
is a colour-preserving bijection. Now combine $\psi$ with the individual $\psi_\ell$ to obtain $\cl(\psi)$. We then put  
$$\psi_{G'_{n+1}} = \cl(\psi) \restriction G'_{n+1} \quad \text{and} \quad \psi_{H'_{n+1}} = \cl(\psi) \restriction H'_{n+1}.$$

{\em Defining isomorphisms $\hat{\Pi}_{n+1}$.} To extend the old isomorphisms $\pi_{n,x}$, given by the induction assumption, note that by \ref{i:keepspromises}, $F_{n+1}$ is obtained from $F_n$ by attaching at every leaf in $R'_n$ a copy of the rooted graph $F_{n+1}(\vec{r}_1)$, and similarly by attaching at every leaf in $B'_n$ a copy of the rooted graph $F_{n+1}(\vec{r}_2)$. For each $x \in X_n$ let us write $\hat{\pi}_{n,x}$ for the map sending each $z \in F^G_n$ to the copy of $\pi_{n,x}(z)$ in $\hat{F}^H_n$. By the induction assumption \ref{commutes}, for each $x \in X_n$ the isomorphism
$$ \hat{\pi}_{n,x} \colon F^G_n \rightarrow \hat{F}^H_n$$
preserves the colour of red and blue leaves. Thus, using the maps $g_\ell$ from (\ref{g_ell}), the mappings
$$g_{\hat{\pi}_{n,x}(\ell)}^{-1} \circ g_\ell \colon \cl(N_n)(\vec{q}_\ell) \cong \cl(N_n)(\vec{q}_{\hat{\pi}_{n,x}(\ell)}) $$
are $\cl(\vec{P'})$-respecting isomorphisms of rooted graphs for all $\ell \in R'_n \cup B'_n$. By combining the isomorphism $\pi_{n,x}$ with these isomorphisms between each $F_{n+1}(\vec{q}_{\ell})$ and $\hat{F}_{n+1}(\vec{q}_{\hat{\pi}_{n,x}(\ell)})$, we obtain a $\cl(\vec{P}')$-respecting extension 
\[
\hat{\pi}_{n+1,x} \colon F_{n+1} \rightarrow \hat{F}_{n+1}.
\]
For the new isomorphism $\hat{\pi}_{n+1,v} \colon F_{n+1} \to \hat{F}_{n+1}$, we simply take the `identity' map which extends the map sending each $z \in \tilde{F}_n$ to $\hat{z} \in \hat{\tilde{F}}_n$.

{\em The diagram commutes.} To see that the new diagram above commutes, for each $x \in X_n$ it suffices to check that for all $\ell \in (R_n \cup B_n) \cap V(G_n)$ we have
$$\hat{\pi}_{n+1,x} \circ \psi_{G'_{n+1}} \restriction \script{L}\p{G'_{n+1}(\vec{q}_\ell)} = \psi_{H'_{n+1}} \circ h'_{n+1,x}\restriction \script{L}\p{G'_{n+1}(\vec{q}_\ell)},$$
which by construction of $\cl(\psi)$ above is equivalent to showing that 
$$\hat{\pi}_{n+1,x} \circ \psi_\ell  = \psi_{h_{n,x}(\ell)} \circ h'_{n+1,x}.$$
By definition of $\psi_\ell$ this holds if and only if
$$\hat{\pi}_{n+1,x} \circ g_{\psi(\ell)}^{-1} \circ \alpha^i \circ f_\ell  =g_{\psi(h_{n,x}(\ell))}^{-1} \circ \alpha^i \circ f_{h_{n,x}(\ell)} \circ h'_{n+1,x}.$$
Now by construction of $\hat{\pi}_{n+1,x}$ and  $h'_{n+1,x}$, we have
$$\hat{\pi}_{n+1,x} \circ g_{\psi(\ell)}^{-1} = g_{\hat{\pi}_{n,x}(\psi(\ell))}^{-1} \quad \text{and} \quad f_{h_{n,x}(\ell)} \circ h'_{n+1,x} = f_\ell.$$
Hence, the above is true if and only if
$$g_{\hat{\pi}_{n,x}(\psi(\ell))}^{-1} \circ \alpha^i \circ f_\ell  =g_{\psi(h_{n,x}(\ell))}^{-1} \circ \alpha^i \circ f_\ell.$$
Finally, this last line holds since $\psi(\ell) = \psi_{G_n}(\ell)$ and $\psi(h_{n,x}(\ell)) = \psi_{H_n}(h_{n,x}(\ell))$ by definition of $\psi$, and because 
$$\hat{\pi}_{n,x} \circ \psi_{G_n} (\ell) = \psi_{H_n} \circ h_{n,x}(\ell)$$
by the induction assumption.

For $\hat{\pi}_{n+1,v}$ we see that, as above, it will be sufficient to show that for all $\ell \in (\tilde{R}_n \cup \tilde{B}_n) \cap V(\tilde{G}_n)$ we have
$$\hat{\pi}_{n+1,v} \circ \psi_\ell  = \psi_{h'_{n+1,v}(\ell)} \circ h'_{n+1,v},$$
which reduces as before to showing that,
$$g_{\hat{\pi}_{n+1,v}(\psi(\ell))}^{-1} \circ \alpha^i \circ f_\ell  =g_{\psi(h'_{n+1,v}(\ell))}^{-1} \circ \alpha^i \circ f_\ell.$$
Recall that, $\hat{\pi}_{n+1,v}$ sends each $v$ to $\hat{v}$ and also, since $h'_{n+1,v} \restriction \tilde{G_n} = h$, the image of every leaf $\ell \in (\tilde{R}_n \cup \tilde{B}_n) \cap V(\tilde{G}_n)$ is simply $\hat{l} \in \hat{G}_n(v) \cup \hat{G}_n(r)$. Hence we wish to show that
$$g_{\hat{(\psi(\ell))}}^{-1} \circ \alpha^i \circ f_\ell =g_{\psi(\hat{l})}^{-1} \circ \alpha^i \circ f_\ell,$$
that is, 
$$\hat{(\psi(\ell))} = \psi(\hat{l}),$$
which follows from the construction of $\psi$.
\end{proof}

\subsubsection{Gluing the graphs together}
\label{subsubsec:gluing}
Let us take the cartesian product of $F_{n+1}$ with a ray, which we simply denote by $F_{n+1} \times \mathbb{N}$. If we identify $F_{n+1}$ with the subgraph $F_{n+1} \times \{0\}$, then we can interpret both $\psi_{G'_{n+1}}$ and $\psi_{H'_{n+1}}$ as maps from $\script{L}(G'_{n+1})$ and $\script{L}(H'_{n+1})$ to a set of vertices in $F_{n+1} \times \mathbb{N}$, under the natural isomorphism between $\hat{F}_{n+1}$ and $F_{n+1}$.

Instead of using the function $\psi_{G'_{n+1}}$ directly for our gluing operation, we identify, for every leaf $\ell$ in $\script{L}(G'_{n+1})$ the \emph{unique neighbour} of $\ell$ with $\psi_{G'_{n+1}}(l)$. Formally, define a bijection
$\chi_{G_{n+1}}$ between the neighbours of $\script{L}(G'_{n+1})$ and $\script{L}'(F_{n+1})$ via
\begin{align}
\chi_{G_{n+1}} = \set{(z_1,z_2)}:{\exists \ell \in \script{L}(G'_{n+1}) \text{ s.t. } z_1 \in N(\ell) \text{ and }\psi_{G'_{n+1}} (l) = z_2}, \numberthis \label{chi_G}
\end{align}
and similarly
\begin{align}
\chi_{H_{n+1}} = \set{(z_1,z_2)}:{\exists \ell \in \script{L}(H'_{n+1}) \text{ s.t. } z_1 \in N(\ell) \text{ and }\psi_{H'_{n+1}} (l) = z_2}. \numberthis \label{chi_H}
\end{align}
Since two promise leaves in $G'_{n+1}$ or $H'_{n+1}$ are never adjacent to the same vertex, $\chi_{G_{n+1}}$ and $\chi_{H_{n+1}}$ are indeed bijections. Moreover, since all promise leaves were proper, the vertices in the domain of $\chi_{G_{n+1}}$ and $\chi_{H_{n+1}}$ have degree at least 3.  Using our notion of gluing-sum (see Def.~\ref{gluingsum}), we now define
\begin{align}
G_{n+1} := G'_{n+1} \oplus_{\chi_{G_{n+1}}} (F_{n+1} \times \mathbb{N}) \text{    and     } H_{n+1} := H'_{n+1} \oplus_{\chi_{H_{n+1}}} (F_{n+1} \times \mathbb{N}). \numberthis \label{gluing}
\end{align}

We consider $R_{n+1}$, $B_{n+1}$, $X_{n+1}$ and $Y_{n+1}$ as subsets of $G_{n+1}$ and $H_{n+1}$ in the natural way. Then $\psi_{G_{n+1}}$ and $\psi_{H_{n+1}}$ can be taken to be the maps $\psi_{G'_{n+1}}$ and $\psi_{H'_{n+1}}$, again identifying $\hat{F}_{n+1}$ with $F_{n+1}$ in the natural way. We also take the roots of $G_{n+1}$ and $H_{n+1}$ to be the roots of $G'_{n+1}$ and $H'_{n+1}$ respectively

This completes the construction of graphs $G_{n+1}$, $H_{n+1}$, and $F_{n+1}$, the coloured leaf sets $R_{n+1}, B_{n+1}, R'_{n+1}$, and $B'_{n+1}$, the bijections $\psi_{G_{n+1}}$ and $\psi_{H_{n+1}}$, as well as $\varphi_{n+1} \colon X_{n+1} \to Y_{n+1}$, and $k_{n+1} = 2 (\tilde{k}_n +1)$. In the next section, we show the existence of families of isomorphisms $\script{H}_{n+1}$ and $\Pi_{n+1}$, and verify that \ref{nested}--\ref{commutes} are indeed satisfied for the $(n+1)^{\text{th}}$ instance.

\subsection{The inductive step: verification}
\begin{lemma}
We have $G_{n} \subset G_{n+1}$, $H_n \subset H_{n+1}$, $\Delta(G_{n+1}),\Delta(H_{n+1}) \leq 5$, $\Delta(F_{n+1}) \leq 3$, and the roots of $G_{n+1}$ and $H_{n+1}$ are in $R_{n+1}$ and $B_{n+1}$ respectively.
\end{lemma}
\begin{proof}
We note that $G_n \subset G'_{n+1}$ by construction. Hence, it follows that 
\[
G_n \subset G'_{n+1} \subset G'_{n+1}\oplus_{\chi_{G_{n+1}}} (F_{n+1} \times \mathbb{N}) = G_{n+1},
\]
and similarly for $H_n$. Since we glued together degree 3 and degree 2 vertices, and $\Delta(G_n),\Delta(H_n) \leq 5$ and $\Delta(F_n) \leq 3$, it is clear that the same bounds hold for $n+1$. Finally, since the root of $\tilde{G}_n$ was a placeholder promise, and  $R_{n+1}$ was the corresponding set of promise leaves in $\cl(\tilde{G}_n)$, it follows that the root of $G'_{n+1}$ is in $R_{n+1}$, and hence so is the root of $G_{n+1}$. A similar argument shows that the root of $H_{n+1}$ is in $B_{n+1}$.
\end{proof}

\begin{lemma}
We have $\sigma_0(G_{n+1}) = \sigma_0(H_{n+1}) = k_{n+1}$.
\end{lemma}
\begin{proof}
By construction we have that $\sigma_0(\tilde{G}_n) = \sigma_0(\tilde{H}_n) = k_{n+1}$. Since $G'_{n+1}$ and $H'_{n+1}$ are realised as components of the promise closure of $M_n$, and this was a proper extension, it is a simple check that $\sigma_0(G'_{n+1}) = \sigma_0(H'_{n+1}) = k_{n+1}$. Also note that $F_{n+1} \times \N$ has no maximally bare paths of length bigger than two, since the vertices of degree two in $F_{n+1} \times \N$ are precisely those of the form $(\ell,0)$ with $\ell$ a leaf of $F_{n+1}$.

Since $G'_{n+1} \oplus_{\chi_{G_{n+1}}} (F_{n+1} \times \mathbb{N})$ is formed by gluing a set of degree-two vertices of $F_{n+1} \times \mathbb{N}$ to a set of degree-three vertices in $G'_{n+1}$, it follows that $\sigma_0(G_{n+1}) = k_{n+1}$ as claimed. A similar argument shows that $\sigma_0(H_{n+1}) = k_{n+1}$.
\end{proof}

\begin{lemma}
The graphs $G_{n+1}$ and $H_{n+1}$ are spectrally distinguishable.
\end{lemma}
\begin{proof}
Since in $\tilde{G}_n$ we have that all long maximally bare paths except for those of length $k_{n+1}$ are contained inside $G_n$ or $\hat{H}_n$, it follows from our induction assumption \ref{kbigenough} that $\sigma_1(\tilde{G}_n) = k_n$. However, in $\tilde{H}_n$, we attached $\hat{G}_n(\hat{v})$ to generate an maximally bare path of length $\tilde{k}_n + 1$ in $\tilde{H}_n$ (see Fig.~\ref{constructionfigure2}), implying that  $$\sigma_1(\tilde{H}_n) = \tilde{k}_n + 1 > k_n = \sigma_1(\tilde{G}_n).$$
As before, since the promise closures $G'_{n+1}$ and $H'_{n+1}$ are proper extensions of $\tilde{G}_n$ and $\tilde{H}_n$, they are spectrally distinguishable. Lastly, since $F_{n+1} \times \N$ has no leaves and no maximally bare paths of length bigger than two, the same is true for $G_{n+1}$ and $H_{n+1}$.
\end{proof}

\begin{lemma}\label{l:endstuff}
The graphs $G_{n+1}$ and $H_{n+1}$ have exactly one end, and $\Omega(G_{n+1} \cup H_{n+1}) \subset \overline{R_{n+1} \cup B_{n+1}}$.
\end{lemma}
\begin{proof}
By the induction assumption \ref{dense}, we know that $\Omega(G_n \cup H_n) \subset \overline{R_n \cup B_n}$.  

\begin{claimm}
The set $R_{n+1} \cup B_{n+1}$ is dense for $G'_{n+1}$. 
\end{claimm}
Consider a finite $S \subset V(G'_{n+1})$. We have to show that any infinite component $C$ of $G'_{n+1} - S$ has non-empty intersection with $R_{n+1} \cup B_{n+1}$. 

Let us consider the global structure of $G'_{n+1}$ as being roughly that of an infinite regular tree, as in Figure~\ref{sketchfig2}. Specifically, we imagine a copy of $G_n$ at the top level, at the next level are the copies of $G_n$ and $H_n$ that come from a blue or red leaf in the top level, at the next level are the copies attached to blue or red leaves from the previous level, and so on. 

With this in mind, it is evident that either $C$ contains an infinite component from some copy of $H_n - S$ or $G_n - S$, or $C$ contains an infinite ray from this tree structure. In the first case, we have $\cardinality{C \cap \p{R_n \cup B_n}} = \infty$ by induction assumption. Since any vertex from $R_{n} \cup B_n$ has a leaf from $R_{n+1} \cup B_{n+1}$ within distance $k_{n+1}+1$ (cf.\ Figure \ref{constructionfigure2}), it follows that $C$ also meets $R_{n+1} \cup B_{n+1}$ infinitely often. In the second case, the same conclusion follows, since between each level of our tree structure, there is a pair of leaves in $R_{n+1} \cup B_{n+1}$. This establishes the claim.

\begin{claimm}
The set $R_{n+1} \cup B_{n+1}$ is dense for $H'_{n+1}$. 
\end{claimm}

The proof of the second claim is entirely symmetric to the first claim.

To complete the proof of the lemma, observe that $F_{n+1} \times \mathbb{N}$ is one-ended, and with $R_{n+1} \cup B_{n+1}$, also $\operatorname{dom}(\chi_{G_{n+1}}) \cup \operatorname{dom}(\chi_{H_{n+1}})$ is dense for $G'_{n+1} \cup H'_{n+1}$ by our claims.  So by Corollary~\ref{c:oneend}, the graphs $G_{n+1}$ and $H_{n+1}$ have exactly one end. Moreover, since $R_{n+1} \cup B_{n+1}$ meets both graphs infinitely, it follows immediately that it is dense for $G_{n+1} \cup H_{n+1}$.
\end{proof}

\begin{lemma}
The graph $G_{n+1}$ is a proper bare extension of infinite growth of $G_{n}$ at $R_{n} \cup B_{n}$ to length $k_{n}+1$, and  $\Ball_{G_{n+1}}(G_{n}, k_{n}+1)$ does not meet $R_{n+1} \cup B_{n+1}$. Similarly, $H_{n+1}$ is a proper bare extension of infinite growth of $H_{n}$ at $R_{n} \cup B_{n}$ to length $k_{n}+1$, and $\Ball_{H_{n+1}}(H_{n}, k_{n}+1)$ does not meet $R_{n+1} \cup B_{n+1}$. Hence, \ref{ballsG} and \ref{ballsH} are satisfied at stage $n+1$.
\end{lemma}
\begin{proof}
We will just prove the statement for $G_{n+1}$, as the corresponding proof for $H_{n+1}$ is analogous.

Since $G'_{n+1}$ is an $\p{ (\tilde{R}_n \cup \tilde{B}_n) \cap V(\tilde{G}_n)}$-extension of $\tilde{G}_n$, it follows that $G'_{n+1}$ is an 
\begin{align}
\p{ \p{(\tilde{R}_n \cup \tilde{B}_n) \cap V(G_n)} \cup \rooot{G_n}} = \big((R_n \cup B_n) \cap V(G_n)\big)\text{-extension of } G_n. \numberthis \label{bla123}
\end{align}

However, from the construction of the closure of a graph it is clear that that $G'_{n+1}$ is also an $L'$-extension of the supergraph $K$ of $G_n$ formed by gluing a copy of $\tilde{G}_n(\vec{p_1})$ to every leaf in $R_n \cap V(G_n)$ and a copy of $\tilde{H}_n(\vec{p_2})$ to every leaf in $B_n \cap V(G_n)$, where $L'$ is defined as the set of inherited promise leaves from the copies of $\tilde{G}_n(\vec{p_1})$ and $\tilde{H}_n(\vec{p_2})$. 

However, we note that every promise leaf in $\tilde{G}_n(\vec{p_1})$ and $\tilde{H}_n(\vec{p_2})$ is at distance at least $\tilde{k}_n+1$ from the respective root, and so $\Ball_{G'_{n+1}}(G_n, \tilde{k}_n) = \Ball_{K}(G_n, \tilde{k}_n)$. However, $\Ball_{K}(G_n,\tilde{k}_n)$ can be seen immediately to be an bare extension of $G_n$ at $R_n \cup B_n$ to length $\tilde{k}_n$, and since $\tilde{k}_n \geq k_n + 1$ it follows that $\Ball_{G'_{n+1}}(G_n, k_n + 1)$ is an bare extension of $G_n$ at $R_n \cup B_n$ to length $k_n + 1$ as claimed.

Finally, we note that $R_{n+1} \cup B_{n+1}$ is the set of promise leaves $\cl(\script{L}_n)$. By the same reasoning as before, $\Ball_{G'_{n+1}}(G_n, k_n + 1)$ contains no promise leaf in $\cl(\script{L}_n)$, and so does not meet $R_{n+1} \cup B_{n+1}$ as claimed. Furthermore, it doesn't meet any neighbours of $R_{n+1} \cup B_{n+1}$.

Recall that $G_{n+1}$ is formed by gluing a set of vertices in $(F_{n+1} \times \mathbb{N})$ to neighbours of vertices in $R_{n+1} \cup B_{n+1}$. However, by the above claim, $\Ball_{G'_{n+1}}(G_n, k_n + 1)$ does not meet any of the neighbours of $R_{n+1} \cup B_{n+1}$ and so $\Ball_{G_{n+1}}(G_n, k_n + 1) = \Ball_{G'_{n+1}}(G_n, k_n + 1)$, and the claim follows.

Finally, to see that $G_{n+1}$ is a leaf extension of $G_n$ of infinite growth, it suffices to observe that $G_{n+1} - G_n$ consists of one component only, which is a superset of the infinite graph $F_n \times \N$.
\end{proof}

\begin{lemma}
There is a family of isomorphisms
\[
\script{H}_{n+1} = \set{h_{n+1,x} \colon G_{n+1} - x \to H_{n+1} - \varphi_{n+1}(x)}:{x \in X_{n+1}},
\]
such that
\begin{itemize}
			\item  $h_{n+1, x} \restriction \p{G_{n} - x} = h_{n,x}$ for all $x \in X_{n}$,   
                         \item the image of $R_{n+1} \cap V(G_{n+1})$ under $h_{n+1,x}$ is $R_{n+1} \cap V(H_{n+1})$ for all $x \in X_{n+1}$,
                         \item the image of $B_{n+1} \cap V(G_{n+1})$ under $h_{n+1,x}$ is $B_{n+1} \cap V(H_{n+1})$ for all $x \in X_{n+1}$.
\end{itemize}
\end{lemma}
\begin{proof}
Recall that Lemma \ref{petersclaim} shows that the there exists such a family of isomorphisms between $G'_{n+1}$ and $H'_{n+1}$. Furthermore, we have that
\[
G_{n+1} := G'_{n+1} \oplus_{\chi_{G_{n+1}}} (F_{n+1} \times \mathbb{N}) \text{    and     } H_{n+1} := H'_{n+1} \oplus_{\chi_{H_{n+1}}} (F_{n+1} \times \mathbb{N}).
\]
where it is easy to check that $\chi_{G_{n+1}}$ and $\chi_{H_{n+1}}$ satisfy the assumptions of Lemma \ref{l:extend}, since the functions  $\psi_{G'_{n+1}}$ and $\psi_{H'_{n+1}}$ do by Lemma~\ref{maxclaim}.

More precisely, given $x \in X_{n+1}$ and $h'_{n+1,x}$, it follows from Lemma \ref{maxclaim} that
\[
\chi_{H_{n+1}} \circ h'_{n+1,x} \circ \chi_{G_{n+1}}
\]
extends to an isomorphism $\pi_{n+1,x}$ of $F_{n+1}$. Hence, by Lemma \ref{l:extend}, $h'_{n+1,x}$ extends to an isomorphism $h_{n+1,x}$ from $G_{n+1} - x$ to $H_{n+1}-y$. That this isomorphism satisfies the three properties claimed follows immediately from Lemma \ref{petersclaim} and the fact that $h_{n+1,x}  \restriction \p{G_{n} - x} = h'_{n+1,x}  \restriction \p{G_{n} - x}$.
\end{proof}

\begin{lemma}
There exist bijections
\[
\psi_{G_{n+1}} \colon V(G_{n+1}) \cap \p{R_{n+1} \cup B_{n+1}} \to R'_{n+1} \cup B'_{n+1}
\]
and
\[
\psi_{H_{n+1}} \colon V(H_{n+1}) \cap \p{R_{n+1} \cup B_{n+1}} \to R'_{n+1} \cup B'_{n+1},
\]
and a family of isomorphisms 
\[
\Pi_{n+1} = \set{\pi_{n+1,x} \colon F_{n+1} \to F_{n+1}}:{x \in X_{n+1}},
\]
such that 
\begin{itemize}

     \item $\pi_{n+1,x} \restriction R'_{n+1}$ is a permutation of $R'_{n+1}$ for each $x$, 
     \item $\pi_{n+1,x} \restriction B'_{n+1}$ is a permutation of $B'_{n+1}$ for each $x$, and   
     \item for each $x \in X_{n+1}$, the corresponding diagram commutes: \begin{center}
\begin{tikzpicture}
  \matrix (m)
    [
      matrix of math nodes,
      row sep    = 3em,
      column sep = 9em
    ]
    {
      \script{L}(G_{n+1})              & \script{L}(H_{n+1}) \\
      \script{L}(F_{n+1}) &  \script{L}(F_{n+1})           \\
    };
  \path
    (m-1-1) edge [->] node [left] {\scriptsize{$\psi_{G_{n+1}}$}} (m-2-1)
    (m-1-1.east |- m-1-2)
      edge [->] node [above] {\scriptsize{$h_{n+1,x} \restriction \script{L}(G_{n+1})$}} (m-1-2);
        \path
     (m-2-1) edge [->] node [above] {\scriptsize{$\pi_{n+1,x}  \restriction \script{L}(F_{n+1})$}} (m-2-2);
        \path
     (m-1-2) edge [->] node [right] {\scriptsize{$\psi_{H_{n+1}}$}} (m-2-2);
\end{tikzpicture}
\end{center}   
	I.e.\ for every $\ell \in V(G_{n+1}) \cap \p{R_{n+1} \cup B_{n+1}}$ we have $\pi_{n+1,x}(\psi_{G_{n+1}}(\ell)) = \psi_{H_{n+1}}(h_{n+1,x}(\ell))$.
    \end{itemize}
\end{lemma}
\begin{proof}
Since $R_{n+1},B_{n+1} \subset G'_{n+1} \cup H'_{n+1}$, and $h_{n+1,x}$ extends $h'_{n+1,x}$ for each $x \in X_{n+1}$, this follows immediately from Lemma \ref{maxclaim} after identifying $\hat{F}_{n+1}$ with $F_{n+1}$.
\end{proof}

This completes our recursive construction, and hence the proof of Theorem~\ref{t:one} is complete.

\section{A non-reconstructible graph with countably many ends}\label{s:prooftwo}
In this section we will prove Theorem~\ref{t:count}. Since the proof will follow almost exactly the same argument as the proof of Theorem~\ref{t:one}, we will just indicate briefly here the parts which would need to be changed, and how the proof is structured.

The proof follows the same back and forth construction as in Section \ref{s:bandf}, however instead of starting with finite graphs $G_0$ and $H_0$ we will start with two infinite graphs, each containing one free end. For example we could start with the graphs in Figure \ref{f:countbase}. 
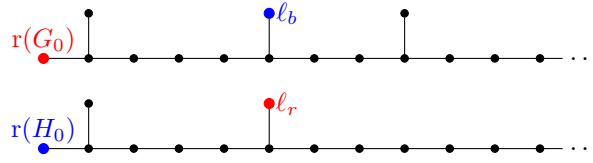
\begin{figure}[ht!]
\centering
\begin{tikzpicture}[scale=0.6]
\def \n {0}
\def \l {1}

\foreach \s in {1,...,12}
{
  \node[draw, circle,scale=.3, fill] (blob\s) at (\s,\n) {};
}

\foreach \s in {1,...,11}
{
	 \pgfmathsetmacro\t{\s + 1}
  \draw (blob\s) -- (blob\t);
}

\foreach \s in {2,6,9}
{
  \node[draw, circle,scale=.3, fill] (leaf\s) at (\s,\l) {};
   \draw (blob\s) -- (leaf\s);
}

\node[draw, red, circle,scale=.4, fill] (redroot) at (1,0) {};
\node[text=red] at (1,.4) {$\rooot{G_0}$};

\node[draw, blue, circle,scale=.4, fill] (blueleaf) at (6,1) {};
\node[text=blue] at (6.4,1) {$\ell_b$};

\def \n {-2}
\def \l {-1}

\foreach \s in {1,...,12}
{
  \node[draw, circle,scale=.3, fill] (blob\s) at (\s,\n) {};
}

\foreach \s in {1,...,11}
{
	 \pgfmathsetmacro\t{\s + 1}
  \draw (blob\s) -- (blob\t);
}

\foreach \s in {2,6}
{
  \node[draw, circle,scale=.3, fill] (leaf\s) at (\s,\l) {};
   \draw (blob\s) -- (leaf\s);
}

\node[draw, blue, circle,scale=.4, fill] (blueroot) at (1,-2) {};
\node[text=blue] at (1,-1.6) {$\rooot{H_0}$};

\node[draw, red, circle,scale=.4, fill] (redleaf) at (6,-1) {};
\node[text=red] at (6.4,-1) {$\ell_r$};

\draw (12,0)--(12.5,0);
\draw (12,-2)--(12.5,-2);

\node at (13,-2) {$\ldots$};

\node at (13,0) {$\ldots$};

\end{tikzpicture}
\caption{A possible choice for $G_0$ and $H_0$, where the dots indicate a ray.}
\label{f:countbase}
\end{figure}

The induction hypotheses remain the same, with the exception of \ref{ends} and \ref{dense} which are replaced by
\begin{enumerate}[label={\upshape($\dagger$\arabic{*}')}]
\setcounter{enumi}{6}
\item \label{ends'} $G_n$ and $H_n$ have exactly one limit end and infinitely many free ends when $n \geq 1$, and
\item \label{dense'} $ \closure{R_n \cup B_n} \cap \Omega(G_n \cup H_n)= \Omega'(G_n \cup H_n)$.
\end{enumerate}

The arguments of Section~\ref{subsec:setup} will then go through mutatis mutandis: for the proof of the analogue of Lemma \ref{l:endstuff}, use Corollary \ref{c:countend} instead of Corollary \ref{c:oneend}.

To show that the construction then yields the desired non-reconstructible pair of graphs with countably many ends, we have to check that \ref{ends'} holds for the limit graphs $G$ and $H$. It is clear that since $ \closure{R_n \cup B_n} \cap \Omega(G_n \cup H_n)= \Omega'(G_n \cup H_n)$, every free end in a graph $G_n$ or $H_n$ remains free in the limit. Moreover, a similar argument to that in Section \ref{subsec:result} shows that any pair of rays in $G$ or $H$ which were not in a free end in some $G_n$ or $H_n$ are equivalent in $G$ or $H$, respectively.

However, since the end space of a locally finite connected graph is a compact metrizable space, and therefore has a countable dense subset, such a graph has at most countably many free ends, since they are isolated in $\Omega(G)$. Hence, both $G$ and $H$ have at most countably many free ends, and one limit end, and so both graphs have countably many ends.

\bibliographystyle{plain}
\bibliography{reconstruction}
\end{document}